\documentclass[twoside,11pt]{article}

\usepackage{blindtext}

\usepackage[abbrvbib, preprint]{jmlr2e}

\usepackage{url}
\usepackage{amsmath}
\usepackage{makecell}
\usepackage{multirow}
\usepackage[singlelinecheck]{caption}
\usepackage[title]{appendix}

\usepackage{lastpage}
\jmlrheading{23}{2022}{1-\pageref{LastPage}}{10/22}{}{00-000}{Elissa Mhanna and Mohamad Assaad}

\ShortHeadings{ZO One-Point Estimate with Distributed Stochastic GT Technique}{Mhanna and Assaad}
\firstpageno{1}

\begin{document}

\title{Zero-Order One-Point Estimate with Distributed Stochastic Gradient-Tracking Technique}

\author{\name Elissa Mhanna \email elissa.mhanna@centralesupelec.fr \\
		\name Mohamad Assaad \email mohamad.assaad@centralesupelec.fr \\
       \addr  Laboratoire des Signaux \& Syst\`{e}mes\\
       CentraleSup\'{e}lec\\
       3 rue Joliot Curie\\
       91190 Gif-sur-Yvette, France}

\editor{My editor}

\maketitle

\begin{abstract}
	In this work, we consider a distributed multi-agent stochastic optimization problem, where each agent holds a local objective function that is smooth and convex, and that is subject to a stochastic process. The goal is for all agents to collaborate to find a common solution that optimizes the sum of these local functions. With the practical assumption that agents can only obtain noisy numerical function queries at exactly one point at a time, we extend the distributed stochastic gradient-tracking method to the bandit setting where we don't have an estimate of the gradient, and we introduce a zero-order (ZO) one-point estimate (1P-DSGT). We analyze the convergence of this novel technique for smooth and convex objectives using stochastic approximation tools, and we prove that it converges almost surely to the optimum. We then study the convergence rate for when the objectives are additionally strongly convex. We obtain a rate of $O(\frac{1}{\sqrt{k}})$ after a sufficient number of iterations $k > K_2$ which is usually optimal for techniques utilizing one-point estimators. We also provide a regret bound of $O(\sqrt{k})$, which is exceptionally good compared to the aforementioned techniques. We further illustrate the usefulness of the proposed technique using numerical experiments.
\end{abstract}


\section{Introduction}

Gradient-free optimization is an old topic in the research community; however, there has been an increased interest recently, especially in machine learning applications, where optimization problems are typically solved with gradient descent algorithms. Successful applications of gradient-free methods in machine learning include competing with an adversary in bandit problems \citep{ref1,ref2-ref}, generating adversarial attacks for deep learning models \citep{ml1,ml2} and reinforcement learning \citep{ml3}.
Gradient-free optimization aims to solve optimization problems with only functional (zero-order) information rather than first-order (FO) gradient information. These techniques are essential in settings where explicit gradient computation may be impractical, expensive, or impossible. Instances of such settings include high data dimensionality, time or resource straining function differentiation, or the cost function not having a closed-form. Zero-order information-based methods include direct search methods \citep{ref3}, 1-point methods \citep{1,ref1,ref6, ml3} where the function is evaluated at a single point with some randomization to estimate the gradient, 2- or more point methods \citep{ref2,ref4,ref5,ref6,zo-dstr,zo-dstr2,ref2-ref, ml1, ml2,ml3}, where functional difference at various points is employed for estimation, and other  methods such as sign information of gradient estimates \citep{ml2}.

Another area of great interest is distributed multi-agent optimization, where agents try to cooperatively solve a problem with information exchange only limited to immediate neighbors in the network. Distributed computing and data storing are particularly essential in fields such as vehicular communications and coordination, data processing and distributed control in sensor networks \citep{app3}, big-data analytics \citep{app2}, and federated learning \citep{app1}. More specifically, one direction of research integrates (sub)gradient-based methods with a consensus/averaging strategy; the local agent incorporates one or multiple consensus steps alongside evaluating the local gradient during optimization. Hence, these algorithms can tackle a fundamental challenge: overcoming differences between agents' local data distributions. 

\subsection{Problem Description}
Consider a set of agents $\mathcal{N} = \{1, 2, \ldots , n\}$ connected by a communication
network.
Each agent $i$ is associated with a local objective function $f_i: \mathbb{R}^d\rightarrow\mathbb{R}$. The global goal of the agents is to collaboratively locate the decision variable $x\in\mathbb{R}^d$ that solves the stochastic optimization problem:
\begin{equation}\label{objective}
	min_{x\in\mathbb{R}^d} \mathcal{F}(x)=  \frac{1}{n}\sum_{i=1}^{n}F_i(x) 
\end{equation}
where
\begin{equation*}
	F_i(x)=\mathbb{E}_S f_i(x,S),
\end{equation*}
with $S\in\mathcal{S}$ denoting an i.i.d. ergodic stochastic process describing uncertainties in the communication system.

We assume that at each time step, agent $i$ can only query the function values of $f_i$ at exactly one point, and can only communicate with its neighbors. Further, we assume that the function queries are noisy $\tilde{f}_{i}=f_{i}+\zeta_{i}$ with $\zeta_{i}$ some additive noise. Agent $i$ must then employ this query to estimate the gradient of the form $g_i(x, S_i)$.

One efficient algorithm with a straightforward averaging scheme to solve this problem is the gradient-tracking (GT) technique, which has proven to achieve rates competing with its centralized counterparts. For example, the acquired error bound under a distributed stochastic variant was found to decrease with the network size n \citep{nedic}. In most work, this technique proved to converge linearly to the optimal solution with constant step size \citep{gr-tr1, gr-tr3,shipu}, which is also a unique attribute among other distributed stochastic gradient algorithms. It has been extended to time-varying (undirected or directed) graphs \citep{gr-tr3}, a gossip-like method which is efficient in communication \citep{nedic}, and nonconvex settings \citep{ref7, gr-tr2, gr-tr6, st-gr-tr2}. All references mentioned above consider the case where an accurate gradient computation or a non-biased gradient estimation with bounded variance (BV) is available.
\subsection{Function Classes}
Consider the following five classes of functions: 
\begin{itemize}
	\item The convex class $\mathcal{C}_{cvx}$ containing all functions $f: \mathbb{R}^d \rightarrow \mathbb{R}$ that are convex.
	
	\item The strongly convex class $\mathcal{C}_{sc}$ containing all functions $f: \mathbb{R}^d \rightarrow \mathbb{R}$ that are continuously differentiable and admit a constant $\mu_f$ such that
	$$\langle\nabla f(x)-\nabla f(y),x-y\rangle\| \geq \mu_f\|x-y\|^2,\;\;\forall x, y \in \mathbb{R}^d.$$
	
	\item The Lipschitz continuous class $\mathcal{C}_{lip}$ containing all functions $f: \mathbb{R}^d \rightarrow \mathbb{R}$ that admit a constant $L_f$ such that
	$$|f(x)-f(y)| \leq L_f\|x-y\|, \;\;\forall x, y \in \mathbb{R}^d.$$
	
	\item The smooth class $\mathcal{C}_{smo}$ containing all functions $f: \mathbb{R}^d \rightarrow \mathbb{R}$ that are continuously differentiable and admit a constant $G_f$ such that
	$$\|\nabla f(x)-\nabla f(y)\| \leq G_f\|x-y\|,\;\;\forall x, y \in \mathbb{R}^d.$$
	\item The gradient dominated class $\mathcal{C}_{gd}$ containing all functions $f: \mathbb{R}^d \rightarrow \mathbb{R}$ that are differentiable, have a global minimizer $x^{*}$, and admit a constant $\nu_f$ such that 
	$$2\nu_f(f(x)-f(x^{*})) \leq \|\nabla f(x)\|^2,\;\;\forall x \in \mathbb{R}^d.$$
\end{itemize}
\begin{table}[t]\footnotesize
	\label{sample-table}
	\begin{center}
		\begin{tabular}{c|c|ccccc}
			\multicolumn{1}{c}{} &\multicolumn{1}{c}{\bf \makecell{GRADIENT\\ESTIMATE}} &\multicolumn{1}{c}{\bf OP} &\multicolumn{1}{c}{\bf \makecell{FUNCTION\\CLASS}}  &\multicolumn{1}{c}{\bf \makecell{CONS-\\ENSUS }} &\multicolumn{1}{c}{\bf \makecell{REGRET\\BOUND}} &\multicolumn{1}{c}{\bf \makecell{CONVERGENCE\\RATE}}
			\\ \hline \\
			\multirow{21}{*}{ZO} &\multirow{3}{*}{One-point} &Cent. 
			&$\mathcal{C}_{cvx}\bigcap\mathcal{C}_{lip}$ &-  &$O(k^{\frac{3}{4}})$ &$O(\frac{1}{\sqrt[4]{k}})$ {\tiny\cite{ref1}} \\
			&\; &Dist.
			&$\mathcal{C}_{sc}\bigcap\mathcal{C}_{lip}\bigcap\mathcal{C}_{smo}$ &None &- &$O(\frac{1}{\sqrt{k}})$ {\tiny\cite{1}} \\
			&\; &Dist. &$\mathcal{C}_{sc}\bigcap\mathcal{C}_{lip}\bigcap\mathcal{C}_{smo}$ &GT &$O(\sqrt{k})$ &$O(\frac{1}{\sqrt{k}})$ {\scriptsize 1P-DSGT} \\ 
			\\ \cline{2-7} \\
			&\multirow{4}{*}{Two-point} &Cent.  &$\mathcal{C}_{cvx}\bigcap\mathcal{C}_{lip}$    &- &$O(\sqrt{k})$ &$O(\frac{1}{\sqrt{k}})$ {\tiny\cite{ref2-ref}} \\
			&\; &Cent.  &$\mathcal{C}_{sc}\bigcap\mathcal{C}_{lip}$    &- &$O(\log k)$ &$O(\frac{\log k}{k})${\tiny\cite{ref2-ref}} \\
			&\; &Dist. &$\mathcal{C}_{lip}\bigcap\mathcal{C}_{smo}$    &None &- &$O(\frac{1}{\sqrt{k}}\log k)$ {\tiny\cite{ref7}} \\
			&\; &Dist.		&$\mathcal{C}_{smo}\bigcap\mathcal{C}_{gd}$    &None &- &$O(\frac{1}{k})$ {\tiny\cite{ref7}} \\
			\\ \cline{2-7} \\
			&(d+1)-point &Cent. &$\mathcal{C}_{sc}\bigcap\mathcal{C}_{lip}\bigcap\mathcal{C}_{smo}$ &-   &$O(\log k)$ &$O(\frac{\log k}{k})$ {\tiny\cite{ref2-ref}} \\
			\\ \cline{2-7} \\
			&\multirow{2}{*}{2d-point} &Dist. &$\mathcal{C}_{smo}$ &GT  &- &$O(\frac{1}{k})$ {\tiny\cite{ref7}}\\
			&\; &Dist.
			&$\mathcal{C}_{smo}\bigcap\mathcal{C}_{gd}$    &GT &- &$O(\lambda^k)$ {\tiny\cite{ref7}} \\
			\\ \cline{2-7} \\
			&Kernel-based &Cent. &$\mathcal{C}_{cvx}\bigcap\mathcal{C}_{lip}$ &-   &$O(\sqrt{k})$ &$O(\frac{1}{\sqrt{k}})$ {\tiny\cite{bubeck}} \\
			\\ \hline \\
			\multirow{3}{*}{FO} &\multirow{3}{*}{Unbiased/BV} &Dist.
			&$\mathcal{C}_{sc}\bigcap\mathcal{C}_{smo}$ &GT   &- &$O(\lambda^k)$ {\tiny\cite{nedic}}; \\
			&\; &\; &\; &\; &\; &{\tiny\cite{st-gr-tr,shipu}}\\
			&\; &Dist.
			&$\mathcal{C}_{smo}$ &GT &-  &$O(\frac{1}{\sqrt{k}})$ {\tiny\cite{st-gr-tr2}} \\
		\end{tabular}
	\end{center} 
\caption{Convergence rates for various algorithms related to our work, classified according to the nature of the gradient estimate, whether the optimization problem (OP) is centralized or distributed, the assumptions on the objective function, whether consensus is aimed at or not, and the achieved regret bound and convergence rate}
\end{table}
\subsection{Related Work}
An adversarial convex bandit problem in a zero-order framework is studied in \citet{ref1, ref2-ref, bubeck}. In \citet{ref1}, a one point estimator is proposed and a regret bound of  $O(k^{\frac{3}{4}})$ is achieved for Lipschitz continuous functions. In \citet{ref2-ref}, they extend the problem to the multi-point setting, where the player queries each loss function at many points; hence, their estimates are unbiased with respect to the true gradient and have bounded variance. When the number of points is two, they prove regret bounds of $\tilde{O}(\sqrt{k})$ with high probability and of $O(\log(k))$ in expectation for strongly convex loss functions. When the number is $d+1$ point, they prove regret bounds of $O(\sqrt{k})$ and of  $O(\log(k))$ with strong convexity.  As mentioned in \citet{ref2-ref} and the references therein, these bounds achieved are even optimal for the full-information setting. In \citet{bubeck},  a kernelized loss estimator is proposed where a generalization of Bernoulli convolutions is adopted, and an annealing schedule for exponential weights is used to control the estimator's variance in a focus region for dimensions higher than 1. \citet{bubeck} achieves a regret bound of  $O(\sqrt{k})$. However, all these references solve the bandit problem in a centralized framework or the case where a single agent has all the data and is the only one responsible for the decision-making. 

\citet{1} consider a distributed stochastic optimization problem in networks under the assumption of smoothness and concavity. They solve it using a stochastic gradient descent method, where the exchange between the nodes is limited to a partial trade of the observation of their local utilities, which each agent then uses to estimate the global utility. With this estimated global utility, each agent can construct a one-point gradient estimator based on a stochastic perturbation. The convergence rate is proved to scale as $O(\frac{1}{\sqrt{k}})$, under the further assumption of a strongly concave objective function. We must remark that in \citet{1}, there is a distributed solution, the optimization variable is a scalar, not a vector, and each agent can only update its own scalar. Hence, there is no consensus to be aimed at, as in our case. In the context of the derivative-free stochastic \emph{centralized} optimization, all \citet{ref2, cv-rate, cv-rate2} have derived several lower bounds to confirm that the convergence rate cannot be better than $O(\frac{1}{\sqrt{k}})$ after k iterations for strongly convex and smooth objective functions.  

All \citet{ gr-tr1, gr-tr2, gr-tr3, gr-tr4, gr-tr5, gr-tr6} present a distributed gradient-tracking method that employs local auxiliary variables to track the average of all agents' gradients, considering the availability of accurate gradient information. In both \citet{gr-tr5, gr-tr6}, each local objective function is an average of finite instantaneous functions. Thus, they incorporate the gradient-tracking algorithm with stochastic averaging gradient technology \citep{gr-tr5} (smooth convex optimization) or with variance reduction techniques \citep{gr-tr6} (smooth nonconvex optimization). At each iteration, they randomly select only one gradient of an instantaneous function to approximate the local batch gradient. In \citet{gr-tr5}, this is an unbiased estimate of the local gradient, whereas, in \citet{gr-tr6}, it is biased. Nevertheless, both references assume access to an exact gradient oracle.

In \citet{ref7}, they develop two algorithms for nonconvex multi-agent optimization. One is a gradient-tracking algorithm based on a noise-free 2d-point estimator of the gradient, achieving a rate of $O(\frac{1}{k})$ with smoothness assumptions and a linear rate for an extra $\nu$-gradient dominated objective assumption. The other is based on a 2-point unbiased estimator without global gradient tracking and achieves a rate of $O(\frac{1}{\sqrt{k}}\log k)$ under Lipschitz continuity and smoothness conditions and $O(\frac{1}{k})$ under an extra gradient dominated function structure. Both these gradient estimators are unbiased with respect to the exact gradient and even have a vanishing variance.

All \citet{nedic, st-gr-tr, shipu, st-gr-tr2} assume access to local stochastic first-order oracles where the gradient is unbiased and with a bounded variance. In the first three, they additionally assume smooth and strongly-convex local objectives, and they all accomplish a linear convergence rate under a constant step size. \citet{nedic} proposes a distributed stochastic gradient-tracking method (DSGT) and a gossip-like stochastic gradient-tracking method (GSGT) where at each round, each agent wakes up with a certain probability. Further, in \citet{nedic}, when the step-size is diminishing, the convergence rate is that of $O(\frac{1}{k})$. \citet{st-gr-tr} employs a gradient-tracking algorithm for agents communicating over a strongly-connected graph. \citet{shipu} introduces a robust gradient-tracking method (R-Push-Pull) in the context of noisy information exchange between agents and with a directed network topology. In \citet{st-gr-tr2}, they propose a gradient-tracking based nonconvex stochastic decentralized (GNSD) algorithm for nonconvex optimization problems in machine learning, and they fulfill a convergence rate of $O(\frac{1}{\sqrt{k}})$ under constant step size.

\subsection{Contributions}
In this paper, we consider smooth and convex local objectives, and we extend the gradient-tracking algorithm to the case where we do not have an estimation of the gradient. Under the realistic assumption that the agent only has access to a single noisy function value at each time, without necessarily knowing the form of this function, we propose a one-point estimator in a stochastic framework. Naturally, one-point estimators are biased with respect to the true gradient and suffer from high variance \citep{z-o} hence they do not match the assumptions for convergence presented in \citet{ref7, nedic, st-gr-tr, shipu, st-gr-tr2}. However, in this work, we analyze and indeed prove the convergence of the algorithm with a biased estimate. We also consider that a stochastic process influences the objective function from one iteration to the other, which is not the case in the aforementioned references. We then study the convergence rate and prove that a rate of $O(\frac{1}{\sqrt{k}})$ after a sufficient number of iterations $k > K_2$ is attainable for smooth and strongly convex objectives. This rate satisfies the lower bounds achieved by its centralized counterparts in the same derivative-free setting \citep{ref2, cv-rate, cv-rate2}. Finally, we show that a regret bound of $O(\sqrt{k})$ is achieved for this algorithm. 

\subsection{Notation}
In all that follows, vectors are column-shaped unless defined otherwise and $\mathbf{1}$ denotes the vector of all entries equal to $1$.  For two vectors $a$, $b$ of the same dimension, $\langle a,b\rangle$ is the inner product. For two matrices $A$, $B\in\mathbb{R}^{n\times d}$, we define
\begin{equation*}
	\langle A,B\rangle =\sum_{i=1}^{n}\langle A_i,B_i\rangle
\end{equation*}
where $A_i$ (respectively, $B_i$) represents the $i$-th row of $A$ (respectively, $B$).  $\|.\|$  denotes	the $2$-norm for vectors and the Frobenius norm for matrices.

We assume that each agent $i$ maintains a local copy $x_i\in\mathbb{R}^d$ of the decision variable and another auxiliary variable $y_i\in\mathbb{R}^d$ and each agent's local function is subject to the stochastic variable $S_i\in\mathbb{R}^m$. At iteration $k$, the respective values are denoted as $x_{i,k}$, $y_{i,k}$, and $S_{i,k}$.  Bold notations denote the concatenated version of the variables, i.e.,
\begin{equation*}
	\begin{split}
		\mathbf{x}:=[x_1, x_2, \ldots , x_n]^T ,\; \mathbf{y}:=[y_1, y_2, \ldots, y_n]^T\in\mathbb{R}^{n\times d}, \;\text{and}\;\mathbf{S}:= [S_1,S_2,\ldots,S_n]^T \in\mathbb{R}^{n\times m}.
	\end{split}
\end{equation*}
We then define the means of the previous two variables as $\bar{x}:=\frac{1}{n}\mathbf{1}^T \mathbf{x}$ and $\bar{y}:=\frac{1}{n}\mathbf{1}^T \mathbf{y} \in\mathbb{R}^{1\times d}$. 

We define the gradient of $F_i$ at the local variable $\nabla F_i(x_i)\in\mathbb{R}^{d}$ and its Hessian matrix $\nabla^2 F_i(x_i)\in\mathbb{R}^{d\times d}$ and we let
\begin{equation*}
	\nabla F(\mathbf{x}):=[\nabla F_1(x_1),\nabla F_2(x_2),\ldots,\nabla F_n(x_n)]^T \in\mathbb{R}^{n\times d}
\end{equation*}
and
\begin{equation*}
	\begin{split}
		\mathbf{g}:=g(\mathbf{x},\mathbf{S}):=[g_1(x_1,S_1),g_2(x_2,S_2),\ldots,g_n(x_n,S_n)]^T \in\mathbb{R}^{n\times d}.
	\end{split}
\end{equation*}
We define its mean $\bar{g} := \frac{1}{n}\mathbf{1}^T \mathbf{g}\in\mathbb{R}^{1\times d}$ and we denote each agent's gradient estimate at time $k$ by $g_{i,k}=g_i(x_{i,k},S_{i,k})$.

\subsection{Basic Assumptions}
In this subsection, we introduce the fundamental assumptions that ensure the performance of the 1P-DSGT algorithm.

\newtheorem{assumption}{Assumption}
\begin{assumption}\label{network}(on the graph)
	The topology of the network is represented by the graph $\mathcal{G} = (\mathcal{N}, \mathcal{E})$ where the edges in $\mathcal{E}\subseteq \mathcal{N} \times \mathcal{N}$ represent communication links. The graph $\mathcal{G}$ is undirected, i.e., $(i, j) \in \mathcal{E}$ iff $(j, i)  \in \mathcal{E}$, and connected (there exists a path of links between any two agents).
	
	$W = [w_{ij}] \in \mathbb{R}^{n\times n}$ denotes the agents' coupling matrix, where
	agents $i$ and $j$ are connected iff $w_{ij} = w_{ji} > 0$ ($w_{ij} = w_{ji} = 0$ otherwise).
	$W$ is a nonnegative matrix and doubly stochastic, i.e., $W \mathbf{1} = \mathbf{1}$ and $\mathbf{1}^T W = \mathbf{1}^T$. All diagonal elements $w_{ii}$ are strictly positive.
\end{assumption} 

\begin{assumption}\label{objective_fct}(on the objective function) 
	We assume the existence and the continuity of both $\nabla F_i(x)$ and $\nabla^2 F_i(x)$. Let $x^*\in\mathbb{R}^{d}$ denote the solution of the problem (\ref{objective}), then
	$\nabla F_i(x^*)=0$ and $\det(\nabla^2 F_i(x^*))>0$, $\forall i \in\mathcal{N}$. To insure the existence of $x^*$, we let the objective function be strictly convex, i.e.,
	\begin{equation}\label{cvx}
		\langle x-x^*,\nabla\mathcal{F}(x)\rangle\geq 0, \forall x\in\mathbb{R}^{d}.
	\end{equation}
	We further assume the boundedness of the local Hessian where there exists a constant $\alpha_1\in\mathbb{R}^+$ such that
	\begin{equation*}
		\|\nabla^2 F_i(x)\|_2\leq\alpha_1, \;\forall i\in\mathcal{N},
	\end{equation*}
	where here it suffices to use the Euclidean norm for matrices (keeping in mind for a matrix $A$, $\|A\|_2 \leq \|A\|_F$).
\end{assumption}
\begin{assumption}\label{local_fcts}
	(on the local functions) All local functions $x\longmapsto f_i(x,S)$ are Lipschitz continuous with Lipschitz constant $L_S$,
	\begin{equation*}
		\|f_i(x,S)-f_i(x',S)\|\leq L_S\|x-x'\|, \;\forall i\in\mathcal{N}.
	\end{equation*}
	In addition, we assume $\mathbb{E}_S f_i (x, S) < \infty$, $\forall i \in\mathcal{N}$, to guarantee the boundedness of the objective $\mathcal{F}(x)$.
\end{assumption}
\begin{assumption}\label{noise}
	(on the additive noise) $\zeta_{i,k}$ is a zero-mean uncorrelated noise with bounded variance, where $E(\zeta_{i,k}) = 0$, $E(\zeta_{i,k}^2)=\alpha_4<\infty$, $\forall i\in\mathcal{N}$, and $E(\zeta_{i,k}\zeta_{j,k}) = 0$ if $i\neq j$.
\end{assumption}

\begin{lemma}\label{rho_w} \citep{gr-tr1} Let $\rho_w$ be the spectral norm $W-\frac{1}{n}\mathbf{1}\mathbf{1}^T$.
	When Assumption \ref{network} is satisfied, we have the following inequality
	$$\|W\omega-\mathbf{1}\bar{\omega}\| \leq\rho_w\|\omega-\mathbf{1}\bar{\omega}\|,\;\forall\omega \in \mathbb{R}^{n\times d} \text{ and } \bar{\omega} = \frac{1}{n}\mathbf{1}^T \omega,$$ and $\rho_w < 1$.
\end{lemma}
\begin{lemma}\label{lipschitz}
	Define $h(\mathbf{x}):=\frac{1}{n}\mathbf{1}^T \nabla F(\mathbf{x}) \in\mathbb{R}^{1\times d}$. Due to the boundedness of the second derivative in Assumption \ref{objective_fct}, the objective function is thus $L$-smooth and we have
	\begin{equation*}
		\|\nabla\mathcal{F}(\bar{x}_k)-h(\mathbf{x}_k)\|\leq \frac{L}{\sqrt{n}}\|\mathbf{x}_k-\mathbf{1}\bar{x}_k\|.
	\end{equation*}
\end{lemma}

\section{Distributed Stochastic Gradient-Tracking Method} \label{algorithm}
We propose to employ a zero-order one-point estimate of the gradient subject to the stochastic process $S$ and an additive noise $\zeta$ while a stochastic perturbation and a step size are introduced, and we assume that each agent can perform this estimation at each iteration. To elaborate, let $g_{i,k}$ denote the aforementioned gradient estimate for agent $i$ at time $k$, then we define it as
\begin{equation}\label{grdt_estimate}
	\begin{split}
		g_{i,k}&=\Phi_{i,k}\tilde{f}_i(x_{i,k}+\gamma_{k}\Phi_{i,k}, S_{i,k})\\
		&=\Phi_{i,k}(f_i(x_{i,k}+\gamma_{k}\Phi_{i,k}, S_{i,k})+\zeta_{i,k}),
	\end{split}
\end{equation}
where $\gamma_{k}>0$ is a vanishing step size and $\Phi_{i,k} \in \mathbb{R}^d$ is a perturbation randomly and independently generated by each agent $i$.
$g_{i,k}$ is in fact a biased estimation of the gradient $\nabla F_i(x_{i,k})$ and
the algorithm can converge under the condition that all parameters are properly chosen. For clarification on the form of this bias and more on the properties of this estimate, refer to \ref{gradient}. 

\subsection{The 1P-DSGT Algorithm}
The following distributed stochastic gradient-tracking method is considered in this part making use of the gradient estimate presented in (\ref{grdt_estimate}). 

Every agent $i$ initializes its variables with an arbitrary value $x_{i,0}$ and $y_{i,0} = g_{i,0}$. Then, at each time $k\in\mathbb{N}$, agent $i$ updates its variables independently according to the following 3 steps:
\begin{equation}\label{algo}
	\begin{split}
		&x_{i,k+1}=\sum_{j=1}^{n}w_{ij}(x_{j,k}-\alpha_k y_{j,k})\\
		&\text{perform the action: } x_{i,k+1}+\gamma_{k+1}\Phi_{i,k+1}\\
		&y_{i,k+1}=\sum_{j=1}^{n}w_{ij}y_{j,k}+g_{i,k+1}-g_{i,k}
	\end{split}
\end{equation}

where $\alpha_k>0$ is a vanishing step size. Algorithm (\ref{algo}) can then be written in the following compact matrix form for clarity of analysis:
\begin{equation}\label{compact}
	\begin{split}
		&\mathbf{x}_{k+1}=W(\mathbf{x}_k-\alpha_k \mathbf{y}_k)\\
		&\text{perform the action: } \mathbf{x}_{k+1} + \gamma_{k+1}\mathbf{\Phi}_{k+1}\\
		&\mathbf{y}_{k+1}=W\mathbf{y}_k+\mathbf{g}_{k+1}-\mathbf{g}_k
	\end{split}
\end{equation}
where $\mathbf{\Phi}_{k} \in \mathbb{R}^{n\times d}$ is defined as $\mathbf{\Phi}_{k}=[\Phi_{1,k},\Phi_{2,k},\ldots,\Phi_{n,k}]^T$.

As is evident from the update of the variables, the exchange between agents is limited to neighboring nodes, and it encompasses the decision variable $\mathbf{x}_{k+1}$ and the auxiliary variable $\mathbf{y}_{k+1}$.

By construction of Algorithm (\ref{compact}), we note that the mean of the auxiliary variable $\mathbf{y}_{k}$ is equal to that of the gradient estimate $\mathbf{g}_{k}$ at every iteration $k$ since $\mathbf{y}_{0}=\mathbf{g}_{0}$, and by recursion, we obtain $\bar{y}_k = \frac{1}{n}\mathbf{1}^T \mathbf{g}_k = \bar{g}_k$.

\begin{assumption}\label{step_sizes}
	(on the step-sizes)	Both $\alpha_k$ and $\gamma_k$ vanish to $0$ as $k\rightarrow\infty$, and satisty the the following sums
	\begin{equation*}
		\sum_{k=1}^{\infty} \alpha_k \gamma_k =\infty,\;\text{and}\; \sum_{k=1}^{\infty} \alpha_k^2<\infty.
	\end{equation*}
\end{assumption}

\begin{assumption}\label{perturbation}
	(on the random perturbation)
	Let $\Phi_{i,k} = (\phi_{i,k}^1, \phi_{i,k}^2, \ldots, \phi_{i,k}^d)^T$.
	
	Each agent $i$ chooses its $\Phi_{i,k}$ vector independently from other agents $j\neq i$. In addition, the elements of $\Phi_{i,k}$ are assumed i.i.d with $\mathbb{E}(\phi_{i,k}^{d_1} \phi_{i,k}^{d_2}) =0$ for $d_1 \neq d_2$ and there exists $\alpha_2 >0$ such that
	$\mathbb{E} (\phi_{i,k}^{d_j})^2 = \alpha_2$, $\forall {d_j}$, $\forall i$.
	We further assume that there exists a constant $\alpha_3 >0$ where
	$\|\Phi_{i,k}\|\leq \alpha_3$, $\forall i.$
\end{assumption}
\begin{example}\label{eg}
	One example is to take $\alpha_k = \alpha_0 (k + 1)^{-\upsilon_1}$ and $\gamma_k = \gamma_0 (k + 1)^{-\upsilon_2}$ with the constants $\alpha_0$, $\gamma_0$, $\upsilon_1$, $\upsilon_2$ $\in \mathbb{R}^+$. As $\sum_{k=1}^{\infty}\alpha_k\gamma_k$ diverges for $\upsilon_1+\upsilon_2\leq 1$ and $\sum_{k=1}^{\infty} \alpha_k^2$ converges for $\upsilon_1 > 0.5$, we can find pairs of $\upsilon_1$ and $\upsilon_2$ so that Assumption \ref{step_sizes} is satisfied.
	
	To achieve the conditions in Assumption \ref{perturbation}, we can choose the probability distribution of $\phi_{i,k}^{d_j}$ to be the symmetrical Bernoulli distribution where $\phi_{i,k}^{d_j} \in \{-\frac{1}{\sqrt{d}},\frac{1}{\sqrt{d}}\}$ with $\mathbb{P}(\phi_{i,k}^{d_j}=-\frac{1}{\sqrt{d}})=\mathbb{P}(\phi_{i,k}^{d_j}=\frac{1}{\sqrt{d}})=0.5$, $\forall d_j$, $\forall i$.
\end{example}

\subsection{Convergence Results}
In this part, we analyze the asymptotic behavior of Algorithm (\ref{compact}). We start the analysis by defining $\mathcal{H}_k$ as the history sequence $\{x_0, y_0, S_0, \ldots, x_{k-1}, y_{k-1}, S_{k-1}, x_k\}$ and denoting by $\mathbb{E}[.|\mathcal{H}_k]$ as the conditional expectation given $\mathcal{H}_k$.

We define $\tilde{g}_{k}$ to be the expected value of $\bar{g}_{k}$ with respect to
all the stochastic terms $S,\Phi,\zeta$ given $\mathcal{H}_k$, i.e.,
\begin{equation*}
	\tilde{g}_{k}=\mathbb{E}_{S,\Phi,\zeta} [\bar{g}_{k}|\mathcal{H}_k]
\end{equation*}
In what follows, we use $\tilde{g}_{k}=\mathbb{E}[\bar{g}_{k}|\mathcal{H}_k]$ for shorthand notation.

We define the error $e_k$ to be the difference between the value of a single realization of $\bar{g}_k$ and its conditional expectation $\tilde{g}_{k}$, i.e.,
\begin{equation*}
	e_k = \bar{g}_k-\tilde{g}_{k},
\end{equation*} 
where $e_k$ can be seen as a stochastic noise. The following lemma describing the vanishing of the stochastic noise is essential for our main result. 

\begin{lemma}\label{martingale}
	If all Assumptions $2-6$ hold and $\|\bar{x}_k\|<\infty$ almost surely, then for any constant $\nu>0$, we have
	\begin{equation*}
		\lim_{K\rightarrow\infty} \mathbb{P}(\sup_{K'\geq K}\|\sum_{k=K}^{K'}\alpha_k e_k\|\geq\nu)=0, \;\forall\nu>0.
	\end{equation*}
	Proof: See \ref{martingale_proof}.
\end{lemma}
For any integer $k\geq 0$, we define the divergence, or the error between the average action taken by the agents $\bar{x}_{k}$ and the optimal solution $x^*$ as
\begin{equation}\label{divergence}
	d_{k}=\|\bar{x}_{k}-x^*\|^2.
\end{equation}
The following theorem describes the main convergence result.
\begin{theorem}\label{cv_th}
	If all Assumptions $1-6$ hold and $\|\mathbf{x}_k\| <\infty$ almost surely, then as $k\rightarrow\infty$, $d_k\rightarrow 0$, $\bar{x}_k\rightarrow x^*$, and $x_{i,k}\rightarrow\bar{x}_k$ for all $i\in\mathcal{N}$ almost surely by applying the Algorithm.
	
	Proof: See \ref{convergence}.
\end{theorem}
\subsection{Convergence Rate}
This part deals with how fast the expected divergence vanishes to find the proposed algorithm's expected convergence rate. To do so, we define the expected divergence to be $D_k = \mathbb{E}[\|\bar{x}_{k}-x^*\|^2]$. The goal is to bound this divergence from above by sequences whose convergence rate is known. The analysis is highly associated with the parameters $\alpha_k$ and $\gamma_k$ that play a significant role in determining this upper bound. Hence, in what follows, the analysis starts with a general form of $\alpha_k$ and $\gamma_k$, then a particular case is considered.

\subsubsection{General Form of $\alpha_k$ and $\gamma_k$} 
We start by considering an additional assumption on the objective function for what follows.
\begin{assumption}\label{strong_convexity}
	Let $\mathcal{F}(x)$ be strongly convex, then there exists $\lambda > 0$ such that
	\begin{equation*}
		\langle\nabla\mathcal{F}(x), x-x^*\rangle\geq \lambda \|x-x^*\|^2.
	\end{equation*}
\end{assumption}
Our main result regarding the convergence rate is summarized in the following theorem.
\begin{theorem}\label{cv_rate_th} 
	
	Let $R=\| \mathbf{x}_0 - \mathbf{1}\bar{x}_0\|^2$, $\delta_k=(\frac{1+\rho_w^2}{2})^k$, and $\beta_k=\sum_{j=1}^{k} \delta_j\alpha_{k-j}^2$. Let $\bar{M}$ denote the upper bound of $\mathbb{E}[\|\bar{g}_k\|^2]$ and $G$ that of $\|\mathbf{y}_{k}-\mathbf{1}\bar{y}_{k}\|$ and define $\bar{G}=\frac{2\rho_w^2}{1-\rho_w^2} G$ (Refer to \ref{gradient}, \ref{to_use}, and \ref{xx_bar_sum_proof} for proof of boundedness).
	
	Next, we define the constants $A=\lambda\alpha_2$, $B=\alpha_1\alpha_3^3$, $C=\alpha_2\frac{L^2}{\lambda n}$, and 
	\begin{equation*}
		K_0 = \arg\min_{A\alpha_k\gamma_k<1} k.
	\end{equation*}
	We finally define the following parameters:
	\begin{equation}\label{parameters}
		\begin{matrix}
			\kappa_k = \frac{1-(\frac{\gamma_{k+1}}{\gamma_{k}})^2}{\alpha_k\gamma_k}, & \sigma_1 = \underset{k\geq K_0}{\max}\;\kappa_k,&
			\sigma_2 = \underset{k\geq K_0}{\max}\;\frac{\delta_k}{\gamma_k^{2}}, 
			& \sigma_3 = \underset{k\geq K_0}{\max}\; \frac{\beta_k}{\gamma_k^{2}},
			& \sigma_4 = \underset{k\geq K_0}{\max}\; \frac{\alpha_k}{\gamma_k^{3}},\\
			\tau_k = \frac{1-\frac{\alpha_{k+1}\gamma_{k+1}^{-1}}{\alpha_{k}\gamma_{k}^{-1}}}{\alpha_{k}\gamma_k},
			& \sigma_5 = \underset{k\geq K_0}{\max}\;\tau_k,
			& \sigma_6 = \underset{k\geq K_0}{\max}\;\sqrt{\frac{\gamma_k^3}{\alpha_{k}}},
			& \sigma_7 = \underset{k\geq K_0}{\max}\;\frac{\gamma_k}{\alpha_{k}}\delta_k,
			& \sigma_8 = \underset{k\geq K_0}{\max}\;\frac{\gamma_k}{\alpha_{k}}\beta_k.
		\end{matrix}
	\end{equation}
	If $\kappa_k<A$ for any $k\geq K_0$, then
	\begin{equation}\label{rate_1}
		D_k\leq\varsigma_1^2\gamma_k^2,\; \forall k\geq K_0,
	\end{equation}
	with 
	\begin{equation}\label{vartheta}
		\varsigma_1\geq 	\max\Bigg\{\frac{\sqrt{D_{K_0}}}{\gamma_{K_0}}, \frac{B}{2(A-\sigma_1)}+\sqrt{\Big(\frac{B}{2(A-\sigma_1)}\Big)^2+\frac{CR\sigma_2+C\bar{G}\sigma_3 +\bar{M}\sigma_4}{(A-\sigma_1)}}\Bigg\}.
	\end{equation}
	If $\tau_k<A$ for any $k\geq K_0$, then
	\begin{equation}\label{rate_2}
		D_k\leq\varsigma_2^2\frac{\alpha_k}{\gamma_k},\; \forall k\geq K_0,
	\end{equation}
	with 
	\begin{equation}\label{varrho}
		\varsigma_2\geq \max\Bigg\{\sqrt{\frac{D_{K_0}\gamma_{K_0}}{\alpha_{K_0}}},	\frac{B\sigma_6}{2(A-\sigma_5)}+\sqrt{\bigg(\frac{B\sigma_6}{2(A-\sigma_5)}\bigg)^2+\frac{C(R\sigma_7+ \bar{G}\sigma_8)+\bar{M}}{(A-\sigma_5)}}\Bigg\}.
	\end{equation}
	Proof: See \ref{th_rate_proof}.
\end{theorem}

\subsubsection{A Special Case of $\alpha_k$ and $\gamma_k$} 
We now consider the special case mentioned in Example \ref{eg}:
\begin{equation}\label{step_sizes_eg}
	\alpha_k = \alpha_0 (k + 1)^{-\upsilon_1} \text{ and } \gamma_k = \gamma_0 (k + 1)^{-\upsilon_2},
\end{equation}
where $0.5 < \upsilon_1 < 1$ and $0 < \upsilon_2\leq 1-\upsilon_1$.

Before stating the main result, we consider the following lemma.
\begin{lemma}\label{beta_k}
	(Study of $\beta_k$)
	Let $K_1$ be such that
	\begin{equation}\label{K_1}
		K_1 = \arg \min_{\big(\frac{1+\rho_w^2}{2}\big)^k\leq\alpha_k^2} k.
	\end{equation}
	Then, the convergence rate of $\beta_k$ is at least that of order $\frac{1}{k^{3\upsilon_1-1}}$.
	
	Proof: See \ref{proof_beta_k}
\end{lemma}
We then let $K_2$ be such that $K_2\geq\max\{K_0,K_1\}$ and state our next theorem.
\begin{theorem}\label{special_case}
	Let $\alpha_k$ and $\gamma_k$ have the forms given in (\ref{step_sizes_eg}), if $\alpha_0\gamma_0\geq \max\{2\upsilon_2,\upsilon_1-\upsilon_2\}/A$, then we can say that there exists $\Upsilon<\infty$, where
	\begin{equation*}
		D_k \leq \Upsilon(k+1)^{-\min\{2\upsilon_2,\upsilon_1-\upsilon_2\}},\;\forall k\geq K_2.
	\end{equation*}
	Proof: See \ref{special_case_proof}.
\end{theorem}
The parameters clearly affect the upper bound of the convergence rate or rate of expected divergence decay in Theorem \ref{special_case}. As it is evident that
\begin{equation*}
	\max\{2\upsilon_2,\upsilon_1-\upsilon_2\}\leq 0.5,
\end{equation*}
the best choice is when equality holds for $\upsilon_1 = 0.75$ and $\upsilon_2 = 0.25$. With the sufficient condition on the parameters in  Theorem \ref{special_case}, we can finally state that our algorithm converges with a rate of $O(\frac{1}{\sqrt{k}})$ after a sufficient number of iterations $k > K_2$.
\subsection{Regret Bound}
To further examine the performance of our algorithm, we present the following theorem on the achieved regret bound.
\begin{theorem}\label{rgrt}
	Let $\alpha_k$ and $\gamma_k$ have the forms of (\ref{step_sizes_eg}), with $\upsilon_1 = 0.75$ and $\upsilon_2 = 0.25$. Then, the regret bound is given by
	\begin{equation*}
		\mathbb{E}_{\mathcal{H}_k}\bigg[\sum_{k=1}^{K}\mathcal{F}(\bar{x}_k)-\mathcal{F}(x^*)\bigg] \leq \Upsilon L (\sqrt{K+1}-1).
	\end{equation*}
	Proof: See \ref{regret}.
\end{theorem}

\section{Numerical Results}
In this section, we give numerical examples to illustrate the performance of the 1P-DSGT algorithm. We compare it with a general DSGT algorithm based on an unbiased estimator with bounded variance. For this unbiased estimator, we calculate the exact gradient and add white noise to it. The network topology is a connected Erd\H{o}s-Rényi random graph with a probability of $0.3$.

We consider a logistic classification problem to classify $m$ images of the two digits, labeled as $y_{ij} = +1$ or $-1$ from the MNIST data set \citep{mnist}. Each image, $X_{ij}$, is a $785$-dimensional vector and is compressed using a lossy autoencoder to become $10$-dimensional, i.e., $d=10$. The total images are split equally among the agents such that each agent has $m_i = \frac{m}{n}$ images and no access to other ones for privacy constraints. However, the goal is still to make use of all images and to solve collaboratively
\begin{equation*}
	min_{\theta\in\mathbb{R}^d}\frac{1}{n}\sum_{i=1}^{n}\frac{1}{m}\sum_{j=1}^{m_i}\mathbb{E}_{u\sim\mathcal{N}(1,\sigma_u)} \ln(1+\exp(-u_{ij} y_{ij}.X_{ij}^T\theta))+c\|\theta\|^2,
\end{equation*}
while reaching consensus on the decision variable $\theta \in\mathbb{R}^d$. We note here that $u$ models some perturbation on the local querying of every example to add to the randomization of the communication process. 

For the first example, we consider classifying the digits $2$ and $9$ where $m=11907$ images. There are $n=21$ agents in the network and thus each has a local batch of $m_i = 567$ images. We take $\sigma_u = 0.01$ and let $\alpha_k = 0.2(k+1)^{-0.75}$, $\gamma_{k} = 1.3(k+1)^{-0.25}$, and $\Phi_k \in 1.5\times\{-\frac{1}{\sqrt{d}},\frac{1}{\sqrt{d}}\}^d$ with equal probability. Also, every function query is subject to a white noise generated by the standard normal distribution. For the general DSGT algorithm, we let the step size to $\alpha_k = 4(k+1)^{-0.75}$, and we don't consider the perturbation on the objective function nor the noise on the objective function, only the noise on the exact gradient. We let $c=0.1$, and the initialization be the same for both algorithms, with $\theta_{i,0}$ uniformly chosen from $[-0.5,0.5]^d$, $\forall i\in\mathcal{N}$, per instance. We finally average the simulations over $50$ instances. 

The expected evolution of the loss objective function is present in Figure \ref{loss-2-9}. Since the DSGT algorithm with vanishing step size converges at a rate of $O(\frac{1}{k})$ \citep{nedic}, the result we obtain is actually anticipated. Next, at every iteration, we measure the accuracy of the classification against an independent test set of $2041$ images using the updated mean vector $\bar{\theta}_k=\frac{1}{n}\sum_{i=1}^{n}\theta_{i,k}$ of the local decision variables and we present the results in Figure \ref{accuracy-2-9}. 1P-DSGT achieves an accuracy of $95.3\%$ whereas DSGT achieves $97.34\%$ at the final iteration, which are both quite good results and an especially important one for 1P-DSGT considering the difference in the convergence rate.

\begin{minipage}[t]{0.45\textwidth}
	\includegraphics[width=1\textwidth]{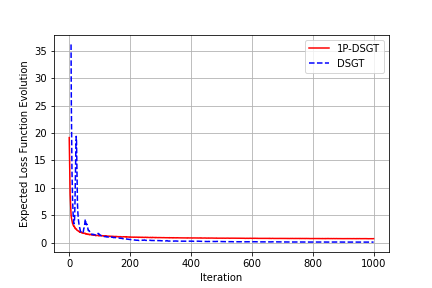}
	\captionof{figure}{Expected loss function evolution of the algorithms 1P-DSGT vs. DSGT considering images of the digits $2$ and $9$.}
	\label{loss-2-9}
\end{minipage}
\hfill
\begin{minipage}[t]{0.45\textwidth}
	\includegraphics[width=1\textwidth]{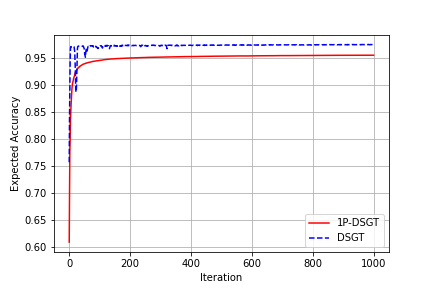}
	\captionof{figure}{Expected accuracy evolution of the algorithms 1P-DSGT vs. DSGT considering images of the digits $2$ and $9$.}
	\label{accuracy-2-9}
\end{minipage}

We display the same result for the last two Figures, \ref{consensus-2-9} and \ref{consensus-z-2-9}; only the second Figure is shown for later iterations to zoom in on the curves. The curves are those of the evolution of the consensus error, or $\sum_{i=1}^{n}\|\theta_{i,k}-\theta_k\|^2$ which is the error between the local decision variables and their average. For both algorithms, the error decreases quite fast. Nonetheless, a swift consensus is reached by 1P-DSGT compared to that of DSGT, which is unpredictable considering all the randomizations. This is an outstanding result for parallel training and computation in networked resource nodes, as consensus no longer raises an issue. 

\begin{minipage}[t]{0.45\textwidth}
	\includegraphics[width=1\textwidth]{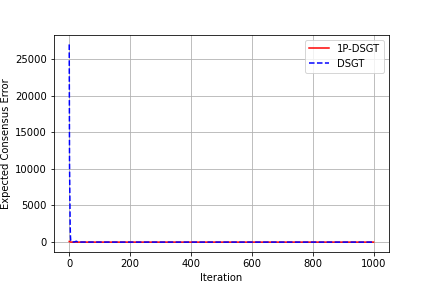}
	\captionof{figure}{Expected consensus error evolution of the algorithms 1P-DSGT vs. DSGT considering images of the digits $2$ and $9$.}
	\label{consensus-2-9}
\end{minipage}
\hfill
\begin{minipage}[t]{0.45\textwidth}
	\includegraphics[width=1\textwidth]{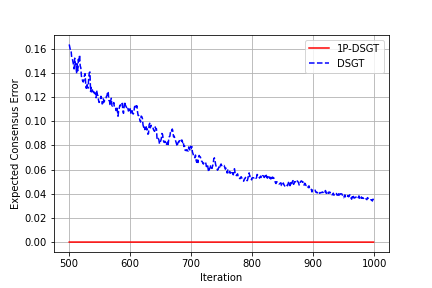}
	\captionof{figure}{Expected consensus error evolution of the algorithms 1P-DSGT vs. DSGT considering images of the digits $2$ and $9$.}
	\label{consensus-z-2-9}
\end{minipage}

The second example is the classification of the images of the digits $3$ and $7$. We have $m=12396$ and we consider $n=6$ agents, thus $m_i=2066$, $\forall i$. Here, we take $\alpha_k = 0.15(k+1)^{-0.75}$, and everything else remains exactly the same as in the previous example. The results are shown in Figures \ref{loss-3-7} through \ref{consensus-z-3-7}. We note that 1P-DSGT achieves $92.46\%$ accuracy, and DSGT achieves $97.44\%$ at the last iteration.

\begin{minipage}[t]{0.45\textwidth}
	\includegraphics[width=1\textwidth]{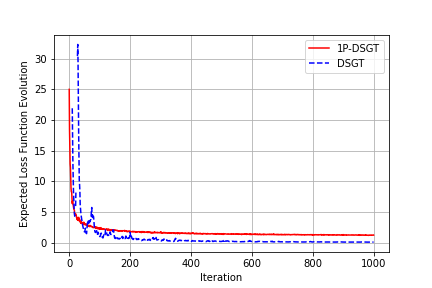}
	\captionof{figure}{Expected loss function evolution of the algorithms 1P-DSGT vs. DSGT considering images of the digits $3$ and $7$.}
	\label{loss-3-7}
\end{minipage}
\hfill
\begin{minipage}[t]{0.45\textwidth}
	\includegraphics[width=1\textwidth]{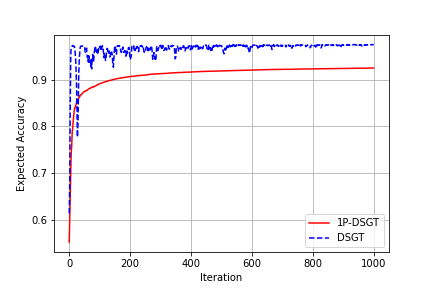}
	\captionof{figure}{Expected accuracy evolution of the algorithms 1P-DSGT vs. DSGT considering images of the digits $3$ and $7$.}
	\label{accuracy-3-7}
\end{minipage}

\begin{minipage}[t]{0.45\textwidth}
	\includegraphics[width=1\textwidth]{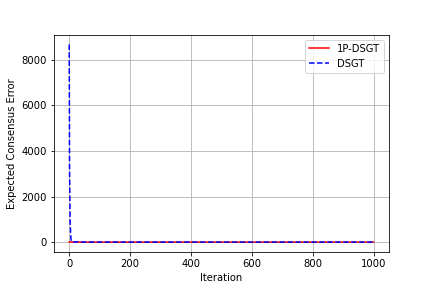}
	\captionof{figure}{Expected consensus error evolution of the algorithms 1P-DSGT vs. DSGT considering images of the digits $3$ and $7$.}
	\label{consensus-3-7}
\end{minipage}
\hfill
\begin{minipage}[t]{0.45\textwidth}
	\includegraphics[width=1\textwidth]{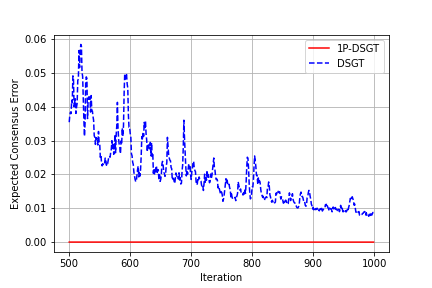}
	\captionof{figure}{Expected consensus error evolution of the algorithms 1P-DSGT vs. DSGT considering images of the digits $3$ and $7$.}
	\label{consensus-z-3-7}
\end{minipage}
\vspace{0.635cm}

One ending remark we can bring attention to is that the 1P-DSGT algorithm is clearly more robust against the distortion brought by the compression of the images. While the DSGT algorithm introduces noticeable fluctuations to all results, 1P-DSGT demonstrates a steady evolution towards the optimum.

\section{Conclusion}
In this work, we extended the gradient-tracking algorithm to present a practical solution to a relevant problem with realistic assumptions. A distributed stochastic gradient-tracking algorithm was studied and proved to converge with a biased and high variance one-point gradient estimate and a stochastic perturbation on the objective function. The convergence rate was proved to be $O(\frac{1}{\sqrt{k}})$ and the regret bound that of $O(\sqrt{k})$. A numerical application confirmed the success and efficiency of the algorithm. For further work, we wish to relax some of the assumptions.




\appendix 
\renewcommand\thesection{\appendixname\ \Alph{section}}
\section{Estimated Gradient}\label{gradient}
In this section, we derive the bias of the gradient estimate with respect to the real gradient of the local objective function. Let
\begin{equation*}
	\breve{g}_{i,k}=\mathbb{E}_{S,\Phi,\zeta} [g_{i,k}|\mathcal{H}_k].
\end{equation*}
Thus,
\begin{equation*}
	\begin{split}
		\breve{g}_{i,k}			&=\mathbb{E}_{S,\Phi,\zeta} [\Phi_{i,k}(f_i(x_{i,k}+\gamma_{k}\Phi_{i,k}, S_{i,k})+\zeta_{i,k})|\mathcal{H}_k]\\
		&=\mathbb{E}_{S,\Phi} [\Phi_{i,k}f_i(x_{i,k}+\gamma_{k}\Phi_{i,k}, S_{i,k})|\mathcal{H}_k]\\
		&=\mathbb{E}_{\Phi} [\Phi_{i,k}F_i(x_{i,k}+\gamma_{k}\Phi_{i,k})|\mathcal{H}_k].
	\end{split}
\end{equation*}
By Taylor’s theorem and the mean-valued theorem, there exists $\tilde{x}_{i,k}$ located between $x_{i,k}$ and $x_{i,k}+\gamma_{k}\Phi_{i,k}$ where
\begin{equation*}
	F_i(x_{i,k}+\gamma_{k}\Phi_{i,k})= F_i(x_{i,k})+\gamma_{k}\langle \Phi_{i,k},\nabla F_i(x_{i,k})\rangle +\frac{\gamma_{k}^2}{2} \langle\Phi_{i,k}, \nabla^2 F_i(\tilde{x}_{i,k})\Phi_{i,k}\rangle, 
\end{equation*}
substituting in the previous definition,
\begin{equation*}
	\begin{split}
		\breve{g}_{i,k} &= F_i(x_{i,k})\mathbb{E}_{\Phi} [\Phi_{i,k}]+\gamma_{k}\mathbb{E}_{\Phi}[\Phi_{i,k} \Phi_{i,k}^T]\nabla F_i(x_{i,k})+ \frac{\gamma_{k}^2}{2} \mathbb{E}_{\Phi}[\Phi_{i,k} \Phi_{i,k}^T \nabla^2  F_i(\tilde{x}_{i,k}) \Phi_{i,k}|\mathcal{H}_k]\\
		&= \alpha_2 \gamma_{k}[\nabla F_i(x_{i,k})+ b_{i,k}].
	\end{split}
\end{equation*}
Thus, the estimation bias has the form
\begin{equation*}
	\begin{split}
		b_{i,k}&= \frac{\breve{g}_{i,k}}{\alpha_2 \gamma_{k}}-\nabla F_i(x_{i,k})\\
		&= \frac{\gamma_{k}}{2 \alpha_2 } \mathbb{E}_{\Phi}[\Phi_{i,k} \Phi_{i,k}^T \nabla^2 F_i(\tilde{x}_{i,k}) \Phi_{i,k}|\mathcal{H}_k].
	\end{split}
\end{equation*}
Let Assumptions \ref{objective_fct} and \ref{perturbation} hold. Then, we can bound the bias as
\begin{equation*}
	\begin{split}
		\|b_{i,k}\| &\leq \frac{\gamma_{k}}{2 \alpha_2 } \mathbb{E}_{\Phi}[ \| \Phi_{i,k}\|_2 \|\Phi_{i,k}^T\|_2 \|\nabla^2 F_i(\tilde{x}_{i,k})\|_2 \|\Phi_{i,k}\|_2|\mathcal{H}_k] \\
		&\leq \gamma_{k} \frac{\alpha_3^3 \alpha_1}{2 \alpha_2 }. 
	\end{split}
\end{equation*}
We can see $\|b_{i,k}\| \rightarrow 0$ as $k\rightarrow \infty$ since $\gamma_{k}$ is vanishing. 	
We remark that 
\begin{equation}\label{exp_g_bar}
	\begin{split}
		\tilde{g}_k &= \mathbb{E}[\bar{g}_k|\mathcal{H}_k]\\
		&= \frac{1}{n}\sum_{i=1}^{n} \mathbb{E}[g_{i,k}|\mathcal{H}_k]\\
		&= \frac{1}{n}\sum_{i=1}^{n}\alpha_2 \gamma_{k}[\nabla F_i(x_{i,k})+ b_{i,k}]\\
		&=\alpha_2 \gamma_{k}[h(\mathbf{x}_k)+ \bar{b}_{k}]
	\end{split}
\end{equation}
is also a biased estimator of $h(\mathbf{x}_k)$ with
\begin{equation}\label{b_bar}
	\begin{split}
		\|\bar{b}_{k}\| &= \|\frac{1}{n}\sum_{i=1}^{n}b_{i,k}\|\\
		&\leq \frac{1}{n}\sum_{i=1}^{n}\|b_{i,k}\|\\
		&\leq \frac{1}{n}\sum_{i=1}^{n} \gamma_{k} \frac{\alpha_3^3 \alpha_1}{2 \alpha_2 }\\
		&=\gamma_{k} \frac{\alpha_3^3 \alpha_1}{2 \alpha_2 }.
	\end{split}	
\end{equation}

\begin{lemma}\label{grdt_bnd}
	Let all Assumptions $2-6$ hold and
	$\|\mathbf{x}_k\|<\infty$ almost surely, then there exists a bounded constant $M > 0$, such that $E[\|\mathbf{g}_{k}\|^2]< M$ almost surely.
\end{lemma}
\begin{proof}
	$\forall i\in\mathcal{N}$, we have
	\begin{equation*}
		\begin{split}
			\mathbb{E}[\|g_{i,k}\|^2|\mathcal{H}_{k}]
			&=\mathbb{E}[\|\Phi_{i,k}(f_i(x_{i,k}+\gamma_{k}\Phi_{i,k}, S_{i,k})+\zeta_{i,k})\|^2|\mathcal{H}_{k}]\\
			&=\mathbb{E}[\|\Phi_{i,k}\|^2\|f_i(x_{i,k}+\gamma_{k}\Phi_{i,k}, S_{i,k})+\zeta_{i,k}\|^2|\mathcal{H}_{k}]\\
			&\overset{(a)}{\leq} \alpha_3^2 \mathbb{E}[(f_i(x_{i,k}+\gamma_{k}\Phi_{i,k}, S_{i,k})+\zeta_{i,k})^2|\mathcal{H}_{k}]\\
			&\overset{(b)}{=}\alpha_3^2\mathbb{E}[f_i^2(x_{i,k}+\gamma_{k}\Phi_{i,k}, S_{i,k})|\mathcal{H}_{k}]+\alpha_3^2\alpha_4\\
			&\overset{(c)}{\leq} \alpha_3^2\mathbb{E}[(\|f_i(0, S_{i,k})\|+L_{S_{i,k}}\|x_{i,k}+\gamma_{k}\Phi_{i,k}\|)^2|\mathcal{H}_{k}]+\alpha_3^2\alpha_4\\
			&\overset{(d)}{\leq} 2\alpha_3^2\mathbb{E}[\mu_{S_{i,k}}^2+L_{S_{i,k}}^2(\|x_{i,k}\|+\gamma_{k}\alpha_3)^2|\mathcal{H}_{k}]+\alpha_3^2\alpha_4\\
			&\overset{(e)}{=} 2\alpha_3^2(\mu+L' (\|x_{i,k}\|+\gamma_{k}\alpha_3)^2)+\alpha_3^2\alpha_4\\
			&<\infty,
		\end{split}
	\end{equation*}
	where $(a)$ is due to Assumption \ref{perturbation}, $(b)$ Assumption \ref{noise}, and $(c)$ Assumption \ref{local_fcts}. We denote $\|f_i(0, S_{i,k})\|=\mu_{S_{i,k}}$ in $(d)$ and the inequality is due to $\frac{x+y}{2}\leq \sqrt{\frac{x^2+y^2}{2}}$, $\forall x, y \in\mathbb{R}$. In $(e)$, $\mu=\mathbb{E}[\mu_{S_{i,k}}^2]$ and $L'=\mathbb{E}[L_{S_{i,k}}^2]$.
\end{proof}
\section{Stochastic Noise}\label{martingale_proof} 
To prove Lemma \ref{martingale}, we begin by demonstrating that the sequence $\{\sum_{k=K}^{K'} \alpha_k e_k\}_{K'\geq K}$ is a martingale. Since $\bar{g}_k$ and $\bar{g}_{k'}$ are independent if $k\neq k'$ and 
\begin{equation*}
	\begin{split}
		\mathbb{E}[e_k]=&\mathbb{E}[\bar{g}_k-\mathbb{E}[\bar{g}_k|\mathcal{H}_k]]\\
		=&\mathbb{E}_{\mathcal{H}_k}[\mathbb{E}[\bar{g}_k-\mathbb{E}[\bar{g}_k|\mathcal{H}_k]|\mathcal{H}_k]]\\
		=&0
	\end{split}
\end{equation*}
by the law of total expectation, the sequence is a martingale. Therefore, for any constant $\nu>0$, we can state
\begin{equation*}
	\begin{split}
		\mathbb{P}(\sup_{K'\geq K}\|\sum_{k=K}^{K'}\alpha_k e_k\|\geq\nu) &\overset{(a)}{\leq} \mathbb{E}(\|\sum_{k=K}^{K'}\alpha_k e_k\|^2)\\
		&=\frac{1}{\nu^2}\mathbb{E}(\sum_{k=K}^{K'}\sum_{k'=K}^{K'}\alpha_k \alpha_{k'}\langle e_k, e_{k'}\rangle )\\
		&\overset{(b)}{=}\frac{1}{\nu^2}\mathbb{E}(\sum_{k=K}^{K'}\|\alpha_k e_k\|^2)\\
		&\leq\frac{1}{\nu^2}\sum_{k=K}^{\infty}\mathbb{E}(\alpha_k^2\|\bar{g}_k-\mathbb{E}[\bar{g}_k|\mathcal{H}_k]\|^2)\\
		&=\frac{1}{\nu^2}\sum_{k=K}^{\infty}\alpha_k^2\mathbb{E}(\|\bar{g}_k\|^2)-\mathbb{E}_{\mathcal{H}_k}(\|\mathbb{E}[\bar{g}_k|\mathcal{H}_k]\|^2)\\
		&\leq\frac{1}{\nu^2}\sum_{k=K}^{\infty}\alpha_k^2\mathbb{E}(\|\bar{g}_k\|^2)\\
		&\overset{(c)}{\leq}\frac{M}{\nu^2}\sum_{k=K}^{\infty}\alpha_k^2,\\
	\end{split}
\end{equation*}
where $(a)$ is due to Doob's martingale inequality \citep{doob}, $(b)$ is since $\mathbb{E}[\langle e_k,e_{k'}\rangle]	= 0$ for any $k\neq k'$, and $(c)$ is by Lemma \ref{grdt_bnd}.

Since $M$ is a bounded constant and $\lim_{K\rightarrow\infty}\sum_{k=K}^{\infty} \alpha_k^2= 0$ by Assumption \ref{step_sizes}, we get $\lim_{K\rightarrow\infty}\frac{M}{\nu^2}\sum_{k=K}^{\infty}\alpha_k^2 =0$ for any bounded constant $\nu$. Hence, the probability that $\|\sum_{k=K}^{K'}\alpha_k e_k\|\geq \nu$ also vanishes as $K\rightarrow\infty$, which concludes the proof.
\section{Proof of Convergence}\label{convergence}
We start by stating the following lemma that will be useful for the proof of convergence.
\begin{lemma}\label{xx_bar_sum}
	If all Assumptions $1-6$ hold and $\|\mathbf{x}_k\|<\infty$ almost surely, then $\lim_{k\rightarrow\infty}\|\mathbf{x}_k-\mathbf{1}\bar{x}_k\|^2 =0 $.
	In fact, we have $\sum_{k=0}^{\infty}\| \mathbf{x}_k-\mathbf{1}\bar{x}_k\|^2<\infty$ as well as
	\begin{equation*}
		\sum_{k=0}^{\infty}\gamma_{k}\alpha_k\| \mathbf{x}_k-\mathbf{1}\bar{x}_k\|<\infty
	\end{equation*}
	almost surely.
	
	Proof: See \ref{xx_bar_sum_proof}.
\end{lemma}
\subsection{Proof of Theorem \ref{cv_th}}\label{to_use}
The goal is to write the divergence in terms of its previous term and to prove that it's finally vanishing.	We know that $\bar{x}_{k+1} = \bar{x}_{k}-\alpha_k \bar{g}_k$.	With this equation, the divergence at time $k+1$ can be written as
\begin{equation*}
	\begin{split}
		d_{k+1}&=\|\bar{x}_{k+1}-x^*\|^2\\
		&=\|\bar{x}_k-\alpha_k\bar{g}_k-x^*\| \\
		&=\|\bar{x}_k-x^*\|^2-2\alpha_k\langle\bar{x}_k-x^*,\bar{g}_k-\mathbb{E}[\bar{g}_k|\mathcal{H}_k]+\mathbb{E}[\bar{g}_k|\mathcal{H}_k]\rangle +\alpha_k^2\|\bar{g}_k\|^2\\
		&=\|\bar{x}_k-x^*\|^2-2\alpha_k\langle\bar{x}_k-x^*,\mathbb{E}[\bar{g}_k|\mathcal{H}_k]\rangle -2\alpha_k\langle\bar{x}_k-x^*,e_k\rangle  +\alpha_k^2\|\bar{g}_k\|^2\\
		&=\|\bar{x}_k-x^*\|^2-2\alpha_k\langle\bar{x}_k-x^*,\mathbb{E}[\bar{g}_k|\mathcal{H}_k]\rangle -2\alpha_k\langle\bar{x}_k-x^*,e_k\rangle  +\alpha_k^2\|\bar{g}_k\|^2\\
		&\overset{(a)}{=} d_k - 2\alpha_2\gamma_{k}\alpha_k\langle\bar{x}_k-x^*,h(\mathbf{x}_k)+\bar{b}_k\rangle -2\alpha_k\langle\bar{x}_k-x^*,e_k\rangle +\alpha_k^2\|\bar{g}_k\|^2\\
		&= d_k - 2\alpha_2\gamma_{k}\alpha_k\langle\bar{x}_k-x^*,\nabla \mathcal{F}(\bar{x}_k)+\bar{b}_k\rangle
		+2\alpha_2\gamma_{k}\alpha_k\langle\bar{x}_k-x^*,\nabla \mathcal{F}(\bar{x}_k)-h(\mathbf{x}_k)\rangle \\
		&\hspace{0.5cm} -2\alpha_k\langle\bar{x}_k-x^*,e_k\rangle+\alpha_k^2\|\bar{g}_k\|^2\\
		&\overset{(b)}{\leq} d_k - 2\alpha_2\gamma_{k}\alpha_k\langle\bar{x}_k-x^*,\nabla \mathcal{F}(\bar{x}_k)+\bar{b}_k\rangle
		+\frac{2\alpha_2 L\gamma_{k}\alpha_k}{\sqrt{n}}\|\bar{x}_k-x^*\|\| \mathbf{x}_k-\mathbf{1}\bar{x}_k\| \\
		&\hspace{0.5cm} -2\alpha_k\langle\bar{x}_k-x^*,e_k\rangle  +\alpha_k^2\|\bar{g}_k\|^2,
	\end{split}
\end{equation*}
where  $(a)$ is due to (\ref{exp_g_bar}) and $(b)$ is due to Lemma \ref{lipschitz}.
By recursion, we have
\begin{equation}\label{div_ineq}
	\begin{split}
		d_{K+1}&\leq d_0 - 2\alpha_2\sum_{k=0}^{K}\gamma_{k}\alpha_k\langle\bar{x}_k-x^*,\nabla \mathcal{F}(\bar{x}_k)+\bar{b}_k\rangle
		+\frac{2\alpha_2 L}{\sqrt{n}}\sum_{k=0}^{K}\gamma_{k}\alpha_k\|\bar{x}_k-x^*\| \| \mathbf{x}_k-\mathbf{1}\bar{x}_k\|\\ 
		&\hspace{0.5cm} -2\sum_{k=0}^{K}\alpha_k\langle\bar{x}_k-x^*,e_k\rangle+\sum_{k=0}^{K}\alpha_k^2\|\bar{g}_k\|^2.\\
	\end{split}
\end{equation}
By Lemma \ref{martingale}, we have $\lim_{K\rightarrow\infty}\|\sum_{k=0}^{K}\alpha_k e_k\|<\infty$ almost surely. Since $\|\bar{x}_k-x^*\|=\|\frac{1}{n}\sum_{i=1}^{n}(x_{i,k}-x^*)\|\leq\frac{1}{n}\sum_{i=1}^{n}\|x_{i,k}-x^*\|<\infty$ almost surely, hence
\begin{equation}
	\lim_{K\rightarrow\infty}\|\sum_{k=0}^{K}\alpha_k\langle\bar{x}_k-x^*,e_k\rangle\|<\infty.
\end{equation}
From Lemma \ref{grdt_bnd}, we have 
\begin{equation}\label{g_bar}
	\|\bar{g}_k\|^2  =\|\frac{1}{n}\sum_{i=1}^{n}g_{i,k}\|^2\leq\frac{n}{n^2}\sum_{i=1}^{n}\|g_{i,k}\|^2 = \frac{1}{n}\sum_{i=1}^{n}\|g_{i,k}\|^2<\infty, \;\text{almost surely.}
\end{equation}
Then, by
Assumption \ref{step_sizes}, 
\begin{equation}\label{sum1}
	\lim_{K\rightarrow\infty}\sum_{k=0}^{K}\alpha_k^2\|\bar{g}_k\|^2<\infty.
\end{equation}
As stated in Lemma \ref{xx_bar_sum}, we have $\sum_{k=0}^{\infty}\gamma_{k}\alpha_k\| \mathbf{x}_k-\mathbf{1}\bar{x}_k\|<\infty$, adding to $\|\bar{x}_k-x^*\| < \infty$ almost surely, then
\begin{equation}\label{sum2}
	\lim_{K\rightarrow\infty}\sum_{k=0}^{K}\gamma_{k}\alpha_k\|\bar{x}_k-x^*\| \| \mathbf{x}_k-\mathbf{1}\bar{x}_k\|<\infty.
\end{equation}
From the above equations (\ref{div_ineq})-(\ref{sum2}), we conclude that there exists $D$ such that $d_{K+1}\leq D+z_K$, with
\begin{equation}\label{z_K}
	z_K=-2\alpha_2\sum_{k=0}^{K}\gamma_{k}\alpha_k\langle\bar{x}_k-x^*,\nabla \mathcal{F}(\bar{x}_k)+\bar{b}_k\rangle.
\end{equation}
Carefully examine (\ref{b_bar}), we can say that there exists $K_b$, such that for $k\geq K_b$, $\bar{b}_k$ becomes very small where we can find an arbitrary small positive value $\epsilon_b$ such that,
\begin{equation}\label{bias_ineq}
	\|\nabla \mathcal{F}(\bar{x}_k)+\bar{b}_k\|\geq (1-\epsilon_b)\|\nabla \mathcal{F}(\bar{x}_k)\|,\; \forall k\geq K_b,
\end{equation}
leading to $-\langle\bar{x}_k-x^*,\nabla \mathcal{F}(\bar{x}_k)+\bar{b}_k\rangle \leq 0$, due to the convexity of $\mathcal{F}$ in (\ref{cvx}). 

Consequently, for any big $K$, $0\leq d_{K+1}< \infty$ and the limit $\lim_{K\rightarrow\infty} d_{K+1} = \bar{d}$ exists.

Thus, there are 2 cases: $\bar{d}>0$ or $\bar{d}=0$.
Assume hypothesis \textit{H1}) $\bar{d}>0$ to  be valid, i.e., $\bar{x}_k$ does not converge to $x^*$, then $\forall\epsilon_h>0$, $\exists K_h$ such that 
\begin{equation*}
	-\langle\bar{x}_k-x^*,\nabla \mathcal{F}(\bar{x}_k)\rangle<-\epsilon_h,\; \forall k\geq K_h.
\end{equation*}
From (\ref{z_K}) and (\ref{bias_ineq}), we get that $\forall k\geq K_m =\max\{K_b,K_h\}$,
\begin{equation*}
	-\langle\bar{x}_k-x^*,\nabla \mathcal{F}(\bar{x}_k)+\bar{b}_k\rangle<-\epsilon_h(1-\epsilon_b),
\end{equation*}
implying
\begin{equation*}
	\begin{split}
		\lim_{K\rightarrow\infty}-\sum_{k=K_m}^{K}\gamma_{k}\alpha_k\langle\bar{x}_k-x^*,\nabla \mathcal{F}(\bar{x}_k)+\bar{b}_k\rangle
		<-\epsilon_h(1-\epsilon_b)\lim_{K\rightarrow\infty}\sum_{k=K_m}^{K}\gamma_{k}\alpha_k<-\infty
	\end{split}
\end{equation*}
since $\sum \gamma_k\alpha_k$ diverges by Assumption \ref{step_sizes}.
As a result, we get $z_K<-\infty$ and $d_{K+1} < -\infty$. However, by definition in (\ref{divergence}), $d_{K+1}>0$. Accordingly, the hypothesis \textit{H1} cannot be true and the case $\bar{d}=0$ is the valid one. We conclude that $\lim_{k\rightarrow\infty}d_{k}=0$,
$\lim_{k\rightarrow\infty}\nabla\mathcal{F}(\bar{x}_k)=0$, and $\lim_{k\rightarrow\infty}\bar{x}_k=x^*$ almost surely.

\subsection{Proof of Lemma \ref{xx_bar_sum}}\label{xx_bar_sum_proof}
We start by replacing the variables with their algorithmic updates.
\begin{equation}\label{x_fct_previous}
	\begin{split}
		\| \mathbf{x}_{k+1}-\mathbf{1}\bar{x}_{k+1}\|^2 &= \| W \mathbf{x}_k -\alpha_k W \mathbf{y}_k - \mathbf{1}\bar{x}_k +\alpha_k \mathbf{1}\bar{y}_k\|^2 \\
		&= \| W \mathbf{x}_k - \mathbf{1}\bar{x}_k\|^2 -2\alpha_k\langle W \mathbf{x}_k - \mathbf{1}\bar{x}_k, W \mathbf{y}_k-\mathbf{1}\bar{y}_k \rangle +\alpha_k^2 \| W \mathbf{y}_k-\mathbf{1}\bar{y}_k \|^2 \\
		&\overset{(a)}{\leq} \| W \mathbf{x}_k - \mathbf{1}\bar{x}_k\|^2 + \alpha_k[\frac{1-\rho_w^2}{2\rho_w^2\alpha_k}\| W \mathbf{x}_k - \mathbf{1}\bar{x}_k\|^2+ \frac{2\rho_w^2\alpha_k}{1-\rho_w^2}\| W \mathbf{y}_k-\mathbf{1}\bar{y}_k \|^2] \\
		&\hspace{0.5cm}+ \alpha_k^2 \|W \mathbf{y}_k-\mathbf{1}\bar{y}_k \|^2\\
		&\overset{(b)}{\leq} \rho_w^2\| \mathbf{x}_k - \mathbf{1}\bar{x}_k\|^2 + \rho_w^2\alpha_k[\frac{1-\rho_w^2}{2\rho_w^2\alpha_k}\| \mathbf{x}_k - \mathbf{1}\bar{x}_k\|^2+ \frac{2\rho_w^2\alpha_k}{1-\rho_w^2}\|\mathbf{y}_k-\mathbf{1}\bar{y}_k \|^2]  \\
		&\hspace{0.5cm}+\rho_w^2 \alpha_k^2 \| \mathbf{y}_k-\mathbf{1}\bar{y}_k \|^2\\
		&= \frac{1+\rho_w^2}{2}\| \mathbf{x}_k - \mathbf{1}\bar{x}_k\|^2+\alpha_k^2\frac{(1+\rho_w^2)\rho_w^2}{1-\rho_w^2}\| \mathbf{y}_k-\mathbf{1}\bar{y}_k \|^2
	\end{split}
\end{equation}
where $(a)$ is by $-2 \epsilon \times \frac{1}{\epsilon}\langle a,b\rangle = -2 \langle \epsilon a,\frac{1}{\epsilon} b\rangle\leq \epsilon^2\|a\|^2 +\frac{1}{\epsilon^2}\|b\|^2$ and $(b)$ is by Lemma \ref{rho_w}. By induction, we have
\begin{equation}\label{xx_bar_squared}
	\begin{split}
		\| \mathbf{x}_{k+1}-\mathbf{1}\bar{x}_{k+1}\|^2 &\leq (\frac{1+\rho_w^2}{2})^{k+1}\| \mathbf{x}_0 - \mathbf{1}\bar{x}_0\|^2+\frac{2\rho_w^2}{1-\rho_w^2}\sum_{j=0}^{k}(\frac{1+\rho_w^2}{2})^{j+1}\alpha_{k-j}^2\| \mathbf{y}_{k-j}-\mathbf{1}\bar{y}_{k-j} \|^2.
	\end{split}
\end{equation}
Since $\sqrt{a+b}<\sqrt{a}+\sqrt{b}$,
\begin{equation}\label{xx_bar}
	\begin{split}
		\| \mathbf{x}_{k+1}-\mathbf{1}\bar{x}_{k+1}\| &\leq (\frac{1+\rho_w^2}{2})^{\frac{k+1}{2}}\| \mathbf{x}_0 - \mathbf{1}\bar{x}_0\|+\sqrt{\frac{2\rho_w^2}{1-\rho_w^2}}\sum_{j=0}^{k}(\frac{1+\rho_w^2}{2})^{\frac{j+1}{2}}\alpha_{k-j}\| \mathbf{y}_{k-j}-\mathbf{1}\bar{y}_{k-j} \|.
	\end{split}
\end{equation}
By repeatedly replacing the variables with the algorithm's iterations, we see that
\begin{equation*}
	\begin{split}
		\mathbf{y}_{k+1} &= W \mathbf{y}_{k}+\mathbf{g}_{k+1}-\mathbf{g}_{k}\\
		&= W (W\mathbf{y}_{k-1}+\mathbf{g}_{k}-\mathbf{g}_{k-1})+\mathbf{g}_{k+1}-\mathbf{g}_{k}\\
		&= W^2 \mathbf{y}_{k-1}-W\mathbf{g}_{k-1}+(W-I)\mathbf{g}_{k}+\mathbf{g}_{k+1}\\
		&= W^2 (W\mathbf{y}_{k-2}+\mathbf{g}_{k-1}-\mathbf{g}_{k-2})-W\mathbf{g}_{k-1}+(W-I)\mathbf{g}_{k}+\mathbf{g}_{k+1}\\
		&= W^3 \mathbf{y}_{k-2}-W^2 \mathbf{g}_{k-2}+ W(W-I)\mathbf{g}_{k-1}+(W-I)\mathbf{g}_{k}+\mathbf{g}_{k+1}\\
		&= W^3 ( W \mathbf{y}_{k-3}+\mathbf{g}_{k-2}-\mathbf{g}_{k-3})-W^2 \mathbf{g}_{k-2}+ W(W-I)\mathbf{g}_{k-1}+(W-I)\mathbf{g}_{k}+\mathbf{g}_{k+1}\\
		&= W^4 \mathbf{y}_{k-3}-W^3 \mathbf{g}_{k-3}+W^2(W-I)\mathbf{g}_{k-2}+ W(W-I)\mathbf{g}_{k-1}+(W-I)\mathbf{g}_{k}+\mathbf{g}_{k+1}\\
		&= \ldots\\
		&= W^{k+1}\mathbf{y}_0 -W^{k}\mathbf{g}_{0}+\sum_{j=0}^{k-1}W^{j}(W-I)\mathbf{g}_{k-j}+\mathbf{g}_{k+1}\\
		&= W^{k}(W-I)\mathbf{g}_0 +\sum_{j=0}^{k-1}W^{j}(W-I)\mathbf{g}_{k-j}+\mathbf{g}_{k+1}\\
		&=\sum_{j=0}^{k}W^{j}(W-I)\mathbf{g}_{k-j}+\mathbf{g}_{k+1},
	\end{split}
\end{equation*}
\begin{equation*}
	\begin{split}
		\mathbf{y}_{k}-\mathbf{1}\bar{y}_{k} &= \sum_{j=0}^{k-1}W^{j}(W-I)\mathbf{g}_{k-1-j}+\mathbf{g}_{k}-\sum_{j=0}^{k-1}\frac{1}{n}\mathbf{1}\mathbf{1}^T W^{j}(W-I)\mathbf{g}_{k-1-j}-\frac{1}{n}\mathbf{1}\mathbf{1}^T\mathbf{g}_{k} \\	
		&= \sum_{j=0}^{k-1}W^{j}(W-I)\mathbf{g}_{k-1-j}+\mathbf{g}_{k}-\sum_{j=0}^{k-1}\frac{1}{n}\mathbf{1}\mathbf{1}^T (W-I)\mathbf{g}_{k-1-j}-\frac{1}{n}\mathbf{1}\mathbf{1}^T\mathbf{g}_{k}\\
		&= \sum_{j=0}^{k-1}(W^{j}-\frac{1}{n}\mathbf{1}\mathbf{1}^T)(W-I)\mathbf{g}_{k-1-j}+\mathbf{g}_{k}-\frac{1}{n}\mathbf{1}\mathbf{1}^T\mathbf{g}_{k}\\
		&= \sum_{j=0}^{k-1}(W-\frac{1}{n}\mathbf{1}\mathbf{1}^T)^{j}(W-I)\mathbf{g}_{k-1-j}+\mathbf{g}_{k}-\mathbf{1}\bar{g}_{k},\\
	\end{split}
\end{equation*}
where the last equality can be proven by recursion and the fact that the matrix $W$ is doubly stochastic by Assumption \ref{network}:

$(W-\frac{1}{n}\mathbf{1}\mathbf{1}^T)^{j+1}=(W^{j}-\frac{1}{n}\mathbf{1}\mathbf{1}^T)(W-\frac{1}{n}\mathbf{1}\mathbf{1}^T) = W^{j+1}-\frac{1}{n}W^j \mathbf{1}\mathbf{1}^T-\frac{1}{n}\mathbf{1}\mathbf{1}^T W+\frac{1}{n}\mathbf{1}\mathbf{1}^T = W^{j+1}-\frac{2}{n}\mathbf{1}\mathbf{1}^T+\frac{1}{n}\mathbf{1}\mathbf{1}^T =W^{j+1}-\frac{1}{n}\mathbf{1}\mathbf{1}^T$.

Thus,
\begin{equation*}
	\begin{split}
		\|\mathbf{y}_{k}-\mathbf{1}\bar{y}_{k}\|&\leq \sum_{j=0}^{k-1}\|(W-\frac{1}{n}\mathbf{1}\mathbf{1}^T)^{j}(W-I)\mathbf{g}_{k-1-j}\|+\|\mathbf{g}_{k}-\mathbf{1}\bar{g}_{k}\|\\
		&\leq \sum_{j=0}^{k-1}\rho_w^j\|(W-I)\mathbf{g}_{k-1-j}\|+\|\mathbf{g}_{k}-\mathbf{1}\bar{g}_{k}\|.\\
	\end{split}
\end{equation*}
From Lemma \ref{grdt_bnd}, we have $\|\mathbf{g}_{k}\|^2<\infty$ almost surely.
\begin{equation*}
	\begin{split}
		\|\mathbf{g}_{k}-\mathbf{1}\bar{g}_{k}\|^2 &= \sum_{i=1}^{n}\|g_{i,k}-\frac{1}{n}\sum_{j=1}^{n}g_{j,k}\|^2 \\
		&= \sum_{i=1}^{n}\bigg(\|g_{i,k}\|^2-2\langle g_{i,k}, \frac{1}{n}\sum_{j=1}^{n}g_{j,k}\rangle+\|\bar{g}_{k}\|^2\bigg)\\
		&= \|\mathbf{g}_{k}\|^2-2n\|\bar{g}_{k}\|^2+n\|\bar{g}_{k}\|^2\\
		&= \|\mathbf{g}_{k}\|^2-n\|\bar{g}_{k}\|^2 \\
		&\leq \|\mathbf{g}_{k}\|^2 \\
		&\leq M'^2<\infty.	\\
	\end{split}
\end{equation*}
Inserting in the previous inequality, we get
\begin{equation}\label{y_limit}
	\begin{split}
		\|\mathbf{y}_{k}-\mathbf{1}\bar{y}_{k}\|
		&\leq \frac{M'}{1-\rho_w}\|(W-I)\|+M'\\
		&= G <\infty,
	\end{split}
\end{equation}
where we have a geometric sum as $\rho_w<1$. 
\begin{enumerate}
	\item \textbf{Proving $\lim_{k\rightarrow\infty}\|\mathbf{x}_k-\mathbf{1}\bar{x}_k\|^2 =0 $}
	
	Reconsider (\ref{x_fct_previous}),
	\begin{equation}\label{fcts_previous}
		\begin{split}
			\| \mathbf{x}_{k+1}-\mathbf{1}\bar{x}_{k+1}\|^2 
			&\leq \frac{1+\rho_w^2}{2}\| \mathbf{x}_k - \mathbf{1}\bar{x}_k\|^2+\alpha_k^2\frac{(1+\rho_w^2)\rho_w^2}{1-\rho_w^2}\| \mathbf{y}_k-\mathbf{1}\bar{y}_k \|^2\\
			\| \mathbf{x}_{k}-\mathbf{1}\bar{x}_{k}\|^2 
			&\leq \frac{1+\rho_w^2}{2}\| \mathbf{x}_{k-1} - \mathbf{1}\bar{x}_{k-1}\|^2+\alpha_{k-1}^2\frac{(1+\rho_w^2)\rho_w^2}{1-\rho_w^2}\| \mathbf{y}_{k-1}-\mathbf{1}\bar{y}_{k-1} \|^2\\
			&\ldots\\
			\| \mathbf{x}_{1}-\mathbf{1}\bar{x}_{1}\|^2 
			&\leq \frac{1+\rho_w^2}{2}\| \mathbf{x}_{0} - \mathbf{1}\bar{x}_{0}\|^2+\alpha_{0}^2\frac{(1+\rho_w^2)\rho_w^2}{1-\rho_w^2}\| \mathbf{y}_{0}-\mathbf{1}\bar{y}_{0} \|^2.\\
		\end{split}
	\end{equation}
	Adding all inequalities in (\ref{fcts_previous}), we obtain
	\begin{equation*}
		\begin{split}
			\| \mathbf{x}_{k+1}-\mathbf{1}\bar{x}_{k+1}\|^2 
			&\leq -\frac{1-\rho_w^2}{2}\sum_{i=1}^{k}\| \mathbf{x}_i - \mathbf{1}\bar{x}_i\|^2+\frac{1+\rho_w^2}{2}\| \mathbf{x}_{0} - \mathbf{1}\bar{x}_{0}\|^2\\
			&\hspace{0.5cm} +\frac{(1+\rho_w^2)\rho_w^2}{1-\rho_w^2}\sum_{i=0}^{k}\alpha_i^2\| \mathbf{y}_i-\mathbf{1}\bar{y}_i \|^2 \\
			&\overset{(a)}{\leq} -\frac{1-\rho_w^2}{2}\sum_{i=1}^{k}\| \mathbf{x}_i - \mathbf{1}\bar{x}_i\|^2+\frac{1+\rho_w^2}{2}\| \mathbf{x}_{0} - \mathbf{1}\bar{x}_{0}\|^2\\
			&\hspace{0.5cm}+G^2 \frac{(1+\rho_w^2)\rho_w^2}{1-\rho_w^2}\sum_{i=0}^{k}\alpha_i^2,
		\end{split}
	\end{equation*}
	
	with $(a)$ being due to (\ref{y_limit}). Let $k\rightarrow\infty$, then the second and third terms are bounded due to Assumption \ref{step_sizes}. There are then 2 cases:  $\sum\| \mathbf{x}_{i} - \mathbf{1}\bar{x}_{i}\|^2$ either diverges or converges.
	Assume the validity of the hypothesis \textit{H2}) $\sum\| \mathbf{x}_{i} - \mathbf{1}\bar{x}_{i}\|^2$ diverges, i.e., $\sum_{i=1}^{\infty}\| \mathbf{x}_{i} - \mathbf{1}\bar{x}_{i}\|^2\rightarrow\infty$. This leads to
	\begin{equation*}
		\| \mathbf{x}_{k+1}-\mathbf{1}\bar{x}_{k+1}\|^2 <-\infty,
	\end{equation*}
	as $-\frac{1-\rho_w^2}{2}<0$. However, $\| \mathbf{x}_{k+1}-\mathbf{1}\bar{x}_{k+1}\|^2 $ should be positive. Thus, hypothesis \textit{H2} cannot be true and $\sum\| \mathbf{x}_{i} - \mathbf{1}\bar{x}_{i}\|^2$ converges. Hence,  $\lim_{k\rightarrow\infty}\|\mathbf{x}_k-\mathbf{1}\bar{x}_k\|^2 =0 $ almost surely.
	
	\item \textbf{Proving $\sum_{k=0}^{\infty}\gamma_{k}\alpha_k\| \mathbf{x}_k-\mathbf{1}\bar{x}_k\|<\infty$}
	
	Going back to (\ref{xx_bar}),
	\begin{equation*}
		\begin{split}
			\| \mathbf{x}_{k+1}-\mathbf{1}\bar{x}_{k+1}\| &\leq (\frac{1+\rho_w^2}{2})^{\frac{k+1}{2}}\| \mathbf{x}_0 - \mathbf{1}\bar{x}_0\|+G\sqrt{\frac{2\rho_w^2}{1-\rho_w^2}}\sum_{j=0}^{k}(\frac{1+\rho_w^2}{2})^{\frac{j+1}{2}}\alpha_{k-j},\\
		\end{split}
	\end{equation*}
	then substituting into the sum $\sum_{k=0}^{\infty}\gamma_{k}\alpha_k\| \mathbf{x}_k-\mathbf{1}\bar{x}_k\|$,
	\begin{equation*}
		\begin{split}
			&\sum_{k=1}^{\infty}\gamma_{k}\alpha_k ((\frac{1+\rho_w^2}{2})^{\frac{k}{2}}\| \mathbf{x}_0 - \mathbf{1}\bar{x}_0\|+G\sqrt{\frac{2\rho_w^2}{1-\rho_w^2}}\sum_{j=0}^{k-1}(\frac{1+\rho_w^2}{2})^{\frac{j+1}{2}}\alpha_{k-1-j})\\
			\leq &\gamma_{0}\alpha_0\| \mathbf{x}_0 - \mathbf{1}\bar{x}_0\|\frac{\sqrt{1+\rho_w^2}}{\sqrt{2}-\sqrt{1+\rho_w^2}}+ G\sqrt{\frac{2\rho_w^2}{1-\rho_w^2}}\sum_{k=1}^{\infty}\gamma_{k}\alpha_k\sum_{j=0}^{k-1}(\frac{1+\rho_w^2}{2})^{\frac{j+1}{2}}\alpha_{k-1-j},\\
		\end{split}
	\end{equation*}
	where the inequality is due to the fact that $\gamma_k$ and $\alpha_k$ are both decreasing step-sizes and we have a geometric sum of ratio $\sqrt{\frac{1+\rho^2}{2}}<1$. We then study the sums in the second term,
	\begin{equation*}
		\begin{split}
			\sum_{k=1}^{\infty}\gamma_{k}\alpha_k\sum_{j=0}^{k-1}(\frac{1+\rho_w^2}{2})^{\frac{j+1}{2}}\alpha_{k-1-j}
			\leq &\sum_{k=1}^{\infty}\gamma_{k}\sum_{j=0}^{k-1}(\frac{1+\rho_w^2}{2})^{\frac{j+1}{2}}\alpha_{k-1-j}^2\\
			= &\sum_{k=1}^{\infty}\gamma_{k}\sum_{j=1}^{k}(\frac{1+\rho_w^2}{2})^{\frac{k-j+1}{2}}\alpha_{j-1}^2\\
			=
			&\sum_{j=1}^{\infty}\alpha_{j-1}^2\sum_{k=j}^{\infty}\gamma_{k}(\frac{1+\rho_w^2}{2})^{\frac{k-j+1}{2}}\\
			\leq
			&\gamma_{0}\sum_{j=1}^{\infty}\alpha_{j-1}^2\sum_{k=j}^{\infty}(\frac{1+\rho_w^2}{2})^{\frac{k-j+1}{2}}\\
			=
			&\gamma_{0}\frac{\sqrt{1+\rho_w^2}}{\sqrt{2}-\sqrt{1+\rho_w^2}}\sum_{j=1}^{\infty}\alpha_{j-1}^2\\
			< &\infty
		\end{split}
	\end{equation*}
	as $\sum\alpha_k^2$ converges by Assumption \ref{step_sizes}.
	
	Finally, $\sum_{k=0}^{\infty}\gamma_{k}\alpha_k\| \mathbf{x}_k-\mathbf{1}\bar{x}_k\|<\infty$.
\end{enumerate}
\section{Convergence Rate}\label{cv_rate}
As a reminder of the definitions of the parameters in Theorem \ref{cv_rate_th}, consider again $R=\| \mathbf{x}_0 - \mathbf{1}\bar{x}_0\|^2$, $\delta_k=(\frac{1+\rho_w^2}{2})^k$, and $\beta_k=\sum_{j=1}^{k} \delta_j\alpha_{k-j}^2$. From Lemma \ref{grdt_bnd} and (\ref{g_bar}), we let $\bar{M}$ denote the upper bound of $\mathbb{E}[\|\bar{g}_k\|^2]$, and from (\ref{y_limit}), $G$ that of $\|\mathbf{y}_{k}-\mathbf{1}\bar{y}_{k}\|$. Let again $\bar{G}=\frac{2\rho_w^2}{1-\rho_w^2} G$.

Our primary result, stated in the following Lemma, is based on finding a relation between two successive iterations of the expected divergence.
\begin{lemma}\label{avg_div} Let $A=\lambda\alpha_2$, $B=\alpha_1\alpha_3^3$, and $ C=\alpha_2\frac{L^2}{\lambda n}$. Whenever Assumption \ref{strong_convexity} holds, for $k>1$, we get
	\begin{equation}\label{avg_div_simp}
		D_{k+1}\leq (1-A\alpha_k\gamma_{k}) D_k +B\alpha_k\gamma_k^2\sqrt{D_k}+C\alpha_k\gamma_k[\delta_k R+\beta_k \bar{G}]+\alpha_k^2 \bar{M}.
	\end{equation}
	Proof: See \ref{proof_exp_div}.
\end{lemma}
Next, we let
\begin{equation*}
	K_0 = \arg\min_{A\alpha_k\gamma_k<1} k.
\end{equation*}
For the ensuing part, the purpose is to locate a vanishing upper bound of $D_k$, making use of the inequality (\ref{avg_div_simp}). The idea is to propose a decreasing sequence $U_{k+1}\leq U_k$ and suppose that $D_k\leq U_k$, $\forall k\geq K_0$, and then verify that $D_{k+1}\leq U_{k+1}$ by induction. The choice of $U_k$ is the most difficult component as one has to keep in mind the general forms of $\alpha_k$ and $\gamma_k$ in (\ref{avg_div_simp}) and what kind of decisions to take regarding these forms, alongside knowing the exact rate of $\beta_k$. An essential property of $U_k$ is presented in the subsequent lemma.
\begin{lemma}\label{U_k}
	If a decreasing sequence $U_{k+1}\leq U_k$ for $k\geq K_0$ exists such that $D_{k+1}\leq U_{k+1}$ can be deduced from $D_k\leq U_k$ and (\ref{avg_div_simp}), then
	\begin{equation}\label{U_k_ineq}
		U_k\geq \Bigg(\frac{B}{2A}\gamma_{k}+\sqrt{\bigg(\frac{B}{2A}\bigg)^2\gamma_{k}^2+\frac{C}{A}[\delta_k R+\beta_k \bar{G}]+\frac{\bar{M}}{A}\frac{\alpha_k}{\gamma_k}}\Bigg)^2.
	\end{equation}
\end{lemma}
An important remark is that the lower bound of $U_k$ in (\ref{U_k_ineq}) is vanishing as $\gamma_{k},\delta_k,\beta_k$, and $\frac{\alpha_k}{\gamma_k}$ are all vanishing. This lower bound provides an insight on the convergence rate of $D_k$ as it cannot be better than that of $\gamma_{k}^2,\delta_k,\beta_k$, and $\frac{\alpha_k}{\gamma_k}$.

Proof: See \ref{proof_U_k}.

The previous Lemma allows us to move forward in confirming the existence of the constants $\varsigma_1$ and $\varsigma_2$ that permit $D_k\leq\varsigma_1^2\gamma_k^2$ and $D_k\leq\varsigma_2^2\frac{\alpha_k}{\gamma_k}$ in Theorem \ref{cv_rate_th}, respectively.
\subsection{Proof of Theorem \ref{cv_rate_th}}\label{th_rate_proof}
\begin{enumerate}
	\item \textbf{Proof of} (\ref{rate_1})
	
	By definition of $\varsigma_1$, $D_{K_0}\leq \varsigma_1^2 \gamma_{K_0}^2$. The next step is to make sure that $D_{k+1}\leq U_{k+1}$ can be obtained from $D_k\leq U_k$, $\forall k\geq K_0$. 	Take $U_k = \varsigma_1^2\gamma_k^2 $, let $D_k\leq U_k$ hold, and substitute in (\ref{avg_div_simp}),
	\begin{equation*}
		D_{k+1}\leq (1-A\alpha_k\gamma_{k}) \gamma_k^2\varsigma_1^2 +B\alpha_k\gamma_k^3\varsigma_1+C\alpha_k\gamma_k[\delta_k R+\beta_k \bar{G}]+\alpha_k^2 \bar{M}.
	\end{equation*}
	We solve $D_{k+1}\leq U_{k+1}$ for $\varsigma_1\in\mathbb{R}^+$
	\begin{equation*}
		(1-A\alpha_k\gamma_{k}) \gamma_k^2\varsigma_1^2 +B\alpha_k\gamma_k^3\varsigma_1+C\alpha_k\gamma_k[\delta_k R+\beta_k \bar{G}]+\alpha_k^2 \bar{M}\leq U_{k+1}=\varsigma_1^2\gamma_{k+1}^2.
	\end{equation*}
	Then, by considering $\kappa_k=\frac{1-(\frac{\gamma_{k+1}}{\gamma_k})^2}{\alpha_k\gamma_k}>0$ as given in (\ref{parameters}),
	\begin{equation*}
		(\kappa_k-A) \varsigma_1^2 	+B\varsigma_1+CR\delta_k \gamma_k^{-2}+C\bar{G}\beta_k\gamma_k^{-2} +\bar{M}\alpha_k\gamma_k^{-3} \leq 0,
	\end{equation*}
	and assuming $\kappa_k - A < 0$, we find a constant $\bar{\varsigma_1}$ such that
	\begin{equation*}
		\varsigma_1\geq\bar{\varsigma_1}=\frac{B}{2(A-\kappa_k)}+\sqrt{\Bigg(\frac{B}{2(A-\kappa_k)}\Bigg)^2+\frac{CR\delta_k \gamma_k^{-2}+C\bar{G}\beta_k\gamma_k^{-2} +\bar{M}\alpha_k\gamma_k^{-3}}{(A-\kappa_k)}},
	\end{equation*}
	keeping in mind that $B$ and $CR\delta_k \gamma_k^{-2}+C\bar{G}\beta_k\gamma_k^{-2} +\bar{M}\alpha_k\gamma_k^{-3}$ are positive by definition.
	Examine the parameters $\sigma_1$, $\sigma_2$, $\sigma_3$ and $\sigma_4$ as they are introduced in (\ref{parameters}), then
	\begin{equation*}
		\bar{\varsigma_1}\leq\frac{B}{2(A-\sigma_1)}+\sqrt{\Bigg(\frac{B}{2(A-\sigma_1)}\Bigg)^2+\frac{CR\sigma_2+C\bar{G}\sigma_3 +\bar{M}\sigma_4}{(A-\sigma_1)}}.
	\end{equation*}
	We conclude that $D_k\leq \varsigma_1^2\gamma_k^2$ where $\varsigma_1$ satisfies the definition (\ref{vartheta}).
	
	\item \textbf{Proof of} (\ref{rate_2})
	
	$D_{K_0}\leq\varsigma_2^2\frac{\gamma_{K_0}}{\alpha_{K_0}}$ by definition of $\varsigma_2$. $\forall k\geq K_0$, let $D_k\leq\varsigma_2^2\frac{\alpha_k}{\gamma_k}$, then
	\begin{equation*}
		D_{k+1}\leq (1-A\alpha_k\gamma_{k}) \varsigma_2^2\frac{\alpha_k}{\gamma_k} +B\alpha_k\gamma_k^2\varsigma_2\sqrt{\frac{\alpha_k}{\gamma_k}}+C\alpha_k\gamma_k[\delta_k R+\beta_k \bar{G}]+\alpha_k^2 \bar{M}.
	\end{equation*}
	Solving $D_{k+1}\leq \varsigma_2^2 \frac{\alpha_{k+1}}{\gamma_{k+1}}$ for $\varsigma_2\in\mathbb{R}^+$,
	\begin{equation*}
		(1-A\alpha_k\gamma_{k}) \varsigma_2^2\frac{\alpha_k}{\gamma_k} +B\alpha_k\gamma_k^2\varsigma_2\sqrt{\frac{\alpha_k}{\gamma_k}}+C\alpha_k\gamma_k[\delta_k R+\beta_k \bar{G}]+\alpha_k^2 \bar{M}\leq \varsigma_2^2 \frac{\alpha_{k+1}}{\gamma_{k+1}}.
	\end{equation*}
	Take $\tau_k = \frac{\frac{\alpha_k}{\gamma_k}-\frac{\alpha_{k+1}}{\gamma_{k+1}}}{\alpha_{k}^2}>0$ as given in (\ref{parameters}), then
	\begin{equation*}
		(\tau_k-A) \varsigma_2^2 +B\gamma_k^{\frac{3}{2}}\alpha_{k}^{-\frac{1}{2}}\varsigma_2+C\gamma_k\alpha_{k}^{-1}[\delta_k R+\beta_k \bar{G}]+\bar{M}\leq 0.
	\end{equation*}
	If $\frac{\alpha_k}{\gamma_k}-\frac{\alpha_{k+1}}{\gamma_{k+1}}<A\alpha_{k}^2$, then $\exists$ $\bar{\varsigma_2}$ such that
	\begin{equation*}
		\varsigma_2\geq\bar{\varsigma_2}=\frac{B\gamma_k^{\frac{3}{2}}\alpha_{k}^{-\frac{1}{2}}}{2(A-\tau_k)}+\sqrt{\bigg(\frac{B\gamma_k^{\frac{3}{2}}\alpha_{k}^{-\frac{1}{2}}}{2(A-\tau_k)}\bigg)^2+\frac{C\gamma_k\alpha_{k}^{-1}[\delta_k R+\beta_k \bar{G}]+\bar{M}}{(A-\tau_k)}}.
	\end{equation*}
	Examine $\sigma_5, \sigma_6, \sigma_7$ and $\sigma_8$ that are defined in (\ref{parameters}), we can say
	\begin{equation*}
		\bar{\varsigma_2}\leq\frac{B\sigma_6}{2(A-\sigma_5)}+\sqrt{\bigg(\frac{B\sigma_6}{2(A-\sigma_5)}\bigg)^2+\frac{C(R\sigma_7+ \bar{G}\sigma_8)+\bar{M}}{(A-\sigma_5)}}.
	\end{equation*}
	We conclude that $D_k\leq \varsigma_2^2\frac{\alpha_k}{\gamma_k}$ with $\varsigma_2$ satisfying (\ref{varrho}).
\end{enumerate}
\subsection{Proof of Lemma \ref{avg_div}}\label{proof_exp_div}
We start by expressing the expected divergence in terms of its previous iteration. 
\begin{equation}\label{main}
	\begin{split}
		D_{k+1} = &\mathbb{E}[\|\bar{x}_{k+1}-x^*\|^2]\\
		= &\mathbb{E}[\|\bar{x}_{k}-\alpha_k \bar{g}_k-x^*\|^2]\\
		= &D_k +\alpha_k^2\mathbb{E}[\|\bar{g}_k\|^2]-2\alpha_k\mathbb{E}[\langle\bar{x}_{k}-x^*, \bar{g}_k\rangle]\\
		\overset{(a)}{=} &D_k +\alpha_k^2\mathbb{E}[\|\bar{g}_k\|^2]-2\alpha_2\alpha_k \gamma_{k}\mathbb{E}[\langle\bar{x}_{k}-x^*, h(\mathbf{x}_k)+ \bar{b}_{k}\rangle]\\
		= &D_k +\alpha_k^2\mathbb{E}[\|\bar{g}_k\|^2]-2\alpha_2\alpha_k \gamma_{k}\mathbb{E}[\langle\bar{x}_{k}-x^*,\mathcal{F}(\bar{x}_k)\rangle]+2\alpha_2\alpha_k \gamma_{k}\mathbb{E}[\langle\bar{x}_{k}-x^*, \mathcal{F}(\bar{x}_k)-h(\mathbf{x}_k)\rangle]\\
		&-2\alpha_2\alpha_k \gamma_{k}\mathbb{E}[\langle\bar{x}_{k}-x^*,\bar{b}_{k}\rangle]\\	
	\end{split}
\end{equation}
where $(a)$ is due to both $\mathbb{E}[e_k|\mathcal{H}_k]=0$ and (\ref{exp_g_bar}):
\begin{equation*}
	\begin{split}
		\mathbb{E}[\langle\bar{x}_{k}-x^*, \bar{g}_k\rangle] &= \mathbb{E}[\langle\bar{x}_{k}-x^*, \bar{g}_k-\mathbb{E}[\bar{g}_k|\mathcal{H}_k]+\mathbb{E}[\bar{g}_k|\mathcal{H}_k]\rangle]\\
		&= \mathbb{E}[\langle\bar{x}_{k}-x^*, e_k\rangle]+\mathbb{E}[\langle\bar{x}_{k}-x^*,\mathbb{E}[\bar{g}_k|\mathcal{H}_k]\rangle]\\
		&= \mathbb{E}_{\mathcal{H}_k}[\mathbb{E}[\langle\bar{x}_{k}-x^*, e_k\rangle|\mathcal{H}_k]]+\mathbb{E}[\langle\bar{x}_{k}-x^*,\mathbb{E}[\bar{g}_k|\mathcal{H}_k]\rangle]\\
		&= 0+\mathbb{E}[\langle\bar{x}_{k}-x^*,\mathbb{E}[\bar{g}_k|\mathcal{H}_k]\rangle].
	\end{split} 
\end{equation*}

From Lemma \ref{grdt_bnd} and (\ref{g_bar}), we have $\mathbb{E}[\|\bar{g}_k\|^2]<\bar{M}$ almost surely, with $\bar{M}=\frac{1}{n}M$ a bounded constant.

By the strong convexity in Assumption \ref{strong_convexity}, we have
\begin{equation}\label{part1}
	\begin{split}
		-2\alpha_2\alpha_k \gamma_{k}\mathbb{E}[\langle\bar{x}_{k}-x^*,\mathcal{F}(\bar{x}_k)\rangle]&\leq -2\lambda\alpha_2\alpha_k \gamma_{k}\mathbb{E}[\|\bar{x}_{k}-x^*\|^2]\\
		&=-2\lambda\alpha_2\alpha_k \gamma_{k} D_k.
	\end{split}
\end{equation}
Next, from Lemma \ref{lipschitz}, we have
\begin{equation*}
	\begin{split}
		2\alpha_2\alpha_k \gamma_{k}\langle\bar{x}_{k}-x^*, \mathcal{F}(\bar{x}_k)-h(\mathbf{x}_k)\rangle
		&\leq 2\alpha_2\alpha_k \gamma_{k}\frac {L}{\sqrt{n}}\|\bar{x}_{k}-x^*\|\| \mathbf{x}_k-\mathbf{1}\bar{x}_k\|\\
		&\overset{(a)}{\leq} \lambda\alpha_2\alpha_k \gamma_{k}\|\bar{x}_{k}-x^*\|^2+\alpha_2\alpha_k \gamma_{k}\frac{ L^2}{\lambda n}\| \mathbf{x}_k-\mathbf{1}\bar{x}_k\|^2,
	\end{split}
\end{equation*}
where $(a)$ is due to $2 \sqrt{\epsilon} \times \frac{1}{\sqrt{\epsilon}}\langle a,b\rangle =2 \langle \sqrt{\epsilon} a,\frac{1}{\sqrt{\epsilon}} b\rangle\leq \epsilon\|a\|^2+\frac{1}{\epsilon}\|b\|^2$.
From (\ref{xx_bar_squared}) and (\ref{y_limit}), we get
\begin{equation*}
	\| \mathbf{x}_k-\mathbf{1}\bar{x}_k\|^2\leq \delta_k R+\beta_k \bar{G}
\end{equation*}
where $R=\| \mathbf{x}_0 - \mathbf{1}\bar{x}_0\|^2$, $\delta_k=(\frac{1+\rho_w^2}{2})^k$, $\beta_k=\sum_{j=1}^{k} \delta_j\alpha_{k-j}^2$, and $\bar{G}=\frac{2\rho_w^2}{1-\rho_w^2} G$.
Hence,
\begin{equation}\label{part2}
	2\alpha_2\alpha_k \gamma_{k}\mathbb{E}[\langle\bar{x}_{k}-x^*, \mathcal{F}(\bar{x}_k)-h(\mathbf{x}_k)\rangle]\leq \lambda\alpha_2\alpha_k \gamma_{k}D_k+\alpha_2\alpha_k \gamma_{k}\frac{L^2}{\lambda n}[\delta_k R+\beta_k \bar{G}].
\end{equation}
From (\ref{b_bar}),
\begin{equation}\label{part3}
	\begin{split}
		-2\alpha_2\alpha_k \gamma_{k}\mathbb{E}[\langle\bar{x}_{k}-x^*,\bar{b}_{k}\rangle]&\leq 2\alpha_2\alpha_k \gamma_{k}\mathbb{E}[\|\bar{x}_{k}-x^*\|\|\bar{b}_{k}\|]\\
		&\leq  \alpha_1\alpha_3^3\alpha_k \gamma_{k}^2 \sqrt{D_k}.
	\end{split}
\end{equation}
Finally, by combining (\ref{main}), (\ref{part1}), (\ref{part2}), and (\ref{part3}), we get (\ref{avg_div_simp}).
\subsection{Proof of Lemma \ref{U_k}}\label{proof_U_k}
Since $1-A\alpha_k\gamma_{k}>0$ when $k\geq K_0$, we may substitute $D_k\leq U_k$ in (\ref{avg_div_simp}),
\begin{equation*}
	D_{k+1}\leq (1-A\alpha_k\gamma_{k}) U_k +B\alpha_k\gamma_k^2\sqrt{U_k}+C\alpha_k\gamma_k[\delta_k R+\beta_k \bar{G}]+\alpha_k^2 \bar{M}.
\end{equation*}
Testing $D_{k+1}\leq U_{k+1}$ in the previous inequality, we get
\begin{equation*}
	(1-A\alpha_k\gamma_{k}) U_k +B\alpha_k\gamma_k^2\sqrt{U_k}+C\alpha_k\gamma_k[\delta_k R+\beta_k \bar{G}]+\alpha_k^2 \bar{M}\leq U_{k+1}\leq U_k
\end{equation*}
\begin{equation}\label{U_k_prb}
	A\gamma_{k} U_k -B\gamma_k^2\sqrt{U_k}-C\gamma_k[\delta_k R+\beta_k \bar{G}]-\alpha_k \bar{M}\geq 0.
\end{equation}
Since $\sqrt{U_k}\geq 0$, the solution of (\ref{U_k_prb}) is given by (\ref{U_k_ineq}), which concludes the proof.

\subsection{Proof of Lemma \ref{beta_k}}\label{proof_beta_k}
We consider the following series, as it affects the maximum possible convergence rate which appears in (\ref{U_k_ineq}) and all inequalities related to $D_k$,
\begin{equation*}
	\beta_k=\sum_{j=1}^{k} \delta_j\alpha_{k-j}^2 = \sum_{j=1}^{k}\bigg(\frac{1+\rho_w^2}{2}\bigg)^j\alpha_{k-j}^2.
\end{equation*}
We then try to find $\beta_{k+1}$ in terms of $\beta_k$, 
\begin{equation*}
	\begin{split}
		\beta_{k+1}=&\sum_{j=1}^{k+1}\bigg(\frac{1+\rho_w^2}{2}\bigg)^j\alpha_{k+1-j}^2 \\
		=&\sum_{j=2}^{k+1}\bigg(\frac{1+\rho_w^2}{2}\bigg)^j\alpha_{k+1-j}^2+\bigg(\frac{1+\rho_w^2}{2}\bigg)\alpha_{k}^2\\ 
		=&\sum_{j=1}^{k}\bigg(\frac{1+\rho_w^2}{2}\bigg)^{j+1}\alpha_{k-j}^2+\bigg(\frac{1+\rho_w^2}{2}\bigg)\alpha_{k}^2\\ 
		=&\bigg(\frac{1+\rho_w^2}{2}\bigg)(\beta_k+\alpha_k^2).
	\end{split}
\end{equation*}
We know that this series has a Q-linear to a Q-sublinear convergence rate as 
\begin{equation*}
	\begin{split}
		\frac{\beta_{k+1}}{\beta_k} &= \bigg(\frac{1+\rho_w^2}{2}\bigg)\Bigg(1+\frac{\alpha_k^2}{\sum_{j=1}^{k}\big(\frac{1+\rho_w^2}{2}\big)^j\alpha_{k-j}^2}\Bigg)\\
		&=\bigg(\frac{1+\rho_w^2}{2}\bigg)+\frac{\alpha_k^2}{\sum_{j=1}^{k}\big(\frac{1+\rho_w^2}{2}\big)^{j-1}\alpha_{k-j}^2}\\
		&=\bigg(\frac{1+\rho_w^2}{2}\bigg)+\frac{\alpha_k^2}{\sum_{j=0}^{k-1}\big(\frac{1+\rho_w^2}{2}\big)^{j}\alpha_{k-1-j}^2}\\
		&=\bigg(\frac{1+\rho_w^2}{2}\bigg)+\frac{1}{\sum_{j=0}^{k-1}\big(\frac{1+\rho_w^2}{2}\big)^{j}\frac{\alpha_{k-1-j}^2}{\alpha_k^2}}\\
		&\overset{(a)}{\leq} \bigg(\frac{1+\rho_w^2}{2}\bigg)+\frac{1}{\sum_{j=0}^{k-1}\big(\frac{1+\rho_w^2}{2}\big)^{j}}\\
		&\overset{(b)}{=} \bigg(\frac{1+\rho_w^2}{2}\bigg)+\frac{1-\big(\frac{1+\rho_w^2}{2}\big)}{1-\big(\frac{1+\rho_w^2}{2}\big)^{k}}\\
	\end{split}
\end{equation*}
where $(a)$ is since $\frac{\alpha_{k-1-j}^2}{\alpha_k^2}\geq 1$ for every $j\in\{0,\ldots,k-1\}$ and $(b)$ is due to the geometric sum $\sum_{j=0}^{k-1}\big(\frac{1+\rho_w^2}{2}\big)^{j} = \frac{1-\big(\frac{1+\rho_w^2}{2}\big)^{k}}{1-\big(\frac{1+\rho_w^2}{2}\big)}$.
Hence,
\begin{equation*}
	\lim_{k\rightarrow\infty}\frac{\beta_{k+1}}{\beta_k}\leq \bigg(\frac{1+\rho_w^2}{2}\bigg)+1-\bigg(\frac{1+\rho_w^2}{2}\bigg)=1
\end{equation*}
since $\lim_{k\rightarrow\infty}\big(\frac{1+\rho_w^2}{2}\big)^k =0$.
Next, to get an idea of how $\beta_k$ converges in terms of $k$, we consider $K_1$ that is defined in (\ref{K_1}).
We know that $K_1$ exists and is finite as $\big(\frac{1+\rho_w^2}{2}\big)^k$ decreases much faster than $\alpha_k^2$.
Taking $\alpha_k$ as that in (\ref{step_sizes_eg}), we know that $0.5<\upsilon_1<1$ and we find the condition on $\upsilon_1$ such that 
\begin{equation*}
	\begin{split}
		\Bigg(\frac{1+\rho_w^2}{2}\Bigg)^k\leq&\alpha_k^2\\
		k\log\Bigg(\frac{1+\rho_w^2}{2}\Bigg)\leq & -2\upsilon_1\log(k+1)\\
		\upsilon_1 \leq &\frac{1}{2}\log\bigg(\frac{2}{1+\rho_w^2}\bigg)\frac{k}{\log(k+1)},\\
	\end{split}
\end{equation*}
which is feasible as $\frac{1}{2}\log\big(\frac{2}{1+\rho_w^2}\big)>0$ and $\frac{k}{\log(k+1)}$ grows very large for $k\geq K_1$, a simple condition is that $1<\frac{1}{2}\log(\frac{2}{1+\rho_w^2})\frac{k}{\log(k+1)}$ which gives 
\begin{equation*}
	K_1 = \arg \min_{2\log^{-1}(\frac{2}{1+\rho_w^2})<k\log^{-1}(k+1)} k.
\end{equation*}
Thus, we can write $\beta_k$ as 
\begin{equation}\label{beta_sum1}
	\begin{split}
		\beta_k &= \sum_{j=1}^{K_1-1}\bigg(\frac{1+\rho_w^2}{2}\bigg)^j\alpha_{k-j}^2 + \sum_{j=K_1}^{k}\bigg(\frac{1+\rho_w^2}{2}\bigg)^j\alpha_{k-j}^2\\
		&\overset{(a)}{\leq} \alpha_{k-K_1}^2\sum_{j=1}^{K_1-1}\bigg(\frac{1+\rho_w^2}{2}\bigg)^j + \sum_{j=K_1}^{k}\alpha_j^2\alpha_{k-j}^2\\
		&\overset{(b)}{=} \alpha_{k-K_1}^2\bigg(\frac{1+\rho_w^2}{1-\rho_w^2}\bigg)\bigg(1-\bigg(\frac{1+\rho_w^2}{2}\bigg)^{K_1-1}\bigg) + \sum_{j=K_1}^{k}\alpha_j^2\alpha_{k-j}^2\\
		&= \beta'\alpha_{k-K_1}^2 + \sum_{j=K_1}^{k}\alpha_j^2\alpha_{k-j}^2,\\
	\end{split}
\end{equation}
where in $(a)$ we used the definition of $K_1$ and the fact that $\alpha_k^2$ is a decreasing step-size, and in $(b)$ the sum of a geometric series.
We know that $\beta'=\big(\frac{1+\rho_w^2}{1-\rho_w^2}\big)\big(1-\big(\frac{1+\rho_w^2}{2}\big)^{K_1-1}\big)$ is finite, then the first sum decreases as $\alpha_{k-K_1}^2=\frac{\alpha_{0}^2}{(k-K_1+1)^{2\upsilon_1}}\sim\frac{1}{k^{2\upsilon_1}}$ with $1<2\upsilon_1<2$, which leaves us with the second sum,
\begin{equation}\label{beta_sum2}
	\begin{split}
		\sum_{j=K_1}^{k}\alpha_j^2\alpha_{k-j}^2 &=\alpha_0^2 \sum_{j=K_1}^{k}\frac{1}{(j+1)^{2\upsilon_1}}\frac{1}{(k-j+1)^{2\upsilon_1}}\\
		&=\alpha_0^2 \sum_{j=K_1+1}^{k+1}\frac{1}{j^{2\upsilon_1}}\frac{1}{(k-j+2)^{2\upsilon_1}}\\
		&=\alpha_0^2 \sum_{j=K_1+1}^{k+1}\frac{1}{j^{2\upsilon_1}}\frac{1}{(k+2)^{2\upsilon_1}(1-\frac{j}{k+2})^{2\upsilon_1}}\\
		&\overset{(a)}{=}\alpha_0^2 \frac{1}{(k+2)^{4\upsilon_1}}\sum_{u\in\mathcal{U}_k}\Big(\frac{1}{u(1-u)}\Big)^{2\upsilon_1}\\
		&=\alpha_0^2 \frac{|\mathcal{U}_k|}{(k+2)^{4\upsilon_1}}\sum_{u\in\mathcal{U}_k}\Big(\frac{1}{u^2(1-u)^2}\Big)^{\upsilon_1}\times\frac{1}{|\mathcal{U}_k|}\\
		&\overset{(b)}{\leq}\alpha_0^2 \frac{|\mathcal{U}_k|}{(k+2)^{4\upsilon_1}}\bigg(\sum_{u\in\mathcal{U}_k}\frac{1}{u^2(1-u)^2}\times\frac{1}{|\mathcal{U}_k|}\bigg)^{\upsilon_1}\\
		&\overset{(c)}{=}\alpha_0^2 \frac{|\mathcal{U}_k|^{1-\upsilon_1}}{(k+2)^{4\upsilon_1}}\bigg(\sum_{u\in\mathcal{U}_k}\frac{1}{u^2}+\frac{1}{(1-u)^2}+\frac{2}{u}+\frac{2}{(1-u)}\bigg)^{\upsilon_1}\\
		&=\alpha_0^2 \frac{(k-K_1+1)^{1-\upsilon_1}}{(k+2)^{4\upsilon_1}}\bigg(\sum_{u\in\mathcal{U}_k}\frac{1}{u^2}+\frac{2}{u}+\sum_{u\in\mathcal{U}_k\setminus\{\frac{k+1}{k+2}\}}\frac{1}{(1-u)^2}+\frac{2}{(1-u)}\\
		&\hspace{8.9cm} +2(k+2)+(k+2)^2\bigg)^{\upsilon_1}\\
		&\overset{(d)}{\leq}\alpha_0^2 \frac{(k-K_1+1)^{1-\upsilon_1}}{(k+2)^{4\upsilon_1}}\bigg((k+2)\int_{\frac{K_1}{k+2}}^{\frac{k+1}{k+2}}\Big(\frac{1}{u^2}+\frac{2}{u}\Big)du\\
		&\hspace{2.1cm}+(k+2)\int_{\frac{K_1+1}{k+2}}^{\frac{k+1}{k+2}}\Big(\frac{1}{(1-u)^2}+\frac{2}{(1-u)}\Big)du
		+2(k+2)+(k+2)^2\bigg)^{\upsilon_1}\\
		&=\alpha_0^2 \frac{(k-K_1+1)^{1-\upsilon_1}}{(k+2)^{3\upsilon_1}}\bigg(\frac{(k+2)(k-K_1+1)}{K_1(k+1)}\\
		&\hspace{3cm}+2\ln\frac{e(k+1)(k-K_1+1)}{K_1}+\frac{(k+2)(k-K_1)}{k-K_1+1}+(k+2)\bigg)^{\upsilon_1}\\
		&\leq\alpha_0^2 \frac{(k-K_1+1)^{1-\upsilon_1}}{(k+2)^{3\upsilon_1}}\bigg(\frac{(k+2)(k+1)}{K_1(k+1)}\\
		&\hspace{3cm}+2\ln\frac{e(k+1)(k-K_1+1)}{K_1}+\frac{(k+2)(k-K_1)}{k-K_1}+(k+2)\bigg)^{\upsilon_1}\\
		&=\alpha_0^2 \frac{(k-K_1+1)^{1-\upsilon_1}}{(k+2)^{3\upsilon_1}}\bigg(2\ln\frac{e(k+1)(k-K_1+1)}{K_1}+(k+2)\Big(2+\frac{1}{K_1}\Big)\bigg)^{\upsilon_1}\\
		&\leq\alpha_0^2 \frac{(k-K_1+1)^{1-\upsilon_1}}{(k+2)^{3\upsilon_1}}\bigg(4\ln\frac{e(k+2)}{K_1}+(k+2)\Big(2+\frac{1}{K_1}\Big)\bigg)^{\upsilon_1}\\
	\end{split}
\end{equation}
\begin{equation*}
	\begin{split}
		&\leq\alpha_0^2 \frac{(k+2)^{1-\upsilon_1}}{(k+2)^{3\upsilon_1}}(k+2)^{\upsilon_1}\bigg(\frac{4}{(k+2)}\ln\frac{e}{K_1}+6+\frac{1}{K_1}\bigg)^{\upsilon_1}\\
		&\leq \beta'' \frac{1}{(k+2)^{3\upsilon_1-1}},\\
	\end{split}
\end{equation*}
where in $(a)$ we changed the summation variable to $u=\frac{j}{k+2}$ and $\mathcal{U}_k=\{\frac{K_1+1}{k+2},\frac{K_1+2}{k+2},\ldots,\frac{k+1}{k+2}\}$, in $(b)$ we used Jensen's inequality $\mathbb{E}[\varphi(x)]\leq\varphi(\mathbb{E}[x])$ for the concave function $\varphi(x) = x^{\upsilon_1}$ with $0.5<\upsilon_1<1$, and $(c)$ is by partial fraction decomposition. In $(d)$, we interpret the sum over $\mathcal{U}_k$ as a Riemann sum in which the function $\frac{1}{u}$ is evaluated at the right endpoint of the interval $[\frac{K_1+i}{k+2},\frac{K_1+i+1}{k+2}]$, for $i=0,1,\ldots,k-K_1$. Since the function $\frac1u$ is monotonically decreasing, it is in fact a \emph{lower} Riemann sum and therefore bounded from above by the integral $\int_{\frac{K_1}{k+2}}^{\frac{k+1}{k+2}}\frac{1}{u}du$. Analogously, we estimate the sum of $\frac{1}{u^2}$ over $\mathcal{U}_k$. Mutatis mutandis, the estimate for the monotonically increasing functions $\frac{1}{1-u}$ and $\frac{1}{(1-u)^2}$ follows.

We have also let $\beta''=\max_{k\geq K_1}\bigg(\frac{4}{(k+2)}\ln\frac{e}{K_1}+6+\frac{1}{K_1}\bigg)^{\upsilon_1} = \bigg(\frac{4}{(K_1+2)}\ln\frac{e}{K_1}+6+\frac{1}{K_1}\bigg)^{\upsilon_1}$. Hence, $\beta'' \frac{1}{(k+2)^{3\upsilon_1-1}}\sim\frac{1}{k^{3\upsilon_1-1}}$ with $0.5<3\upsilon_1-1<2$. 

Since $3\upsilon_1-1<2\upsilon_1$ for $0.5<\upsilon_1<1$, then $\frac{1}{k^{2\upsilon_1}}<\frac{1}{k^{3\upsilon_1-1}}$ for all $k\geq K_1$.

From (\ref{beta_sum1}) and (\ref{beta_sum2}), 
\begin{equation*}
	\begin{split}
		\beta_k &\leq \beta'\frac{1}{(k-K_1+1)^{2\upsilon_1}}+\beta''\frac{1}{(k+2)^{3\upsilon_1-1}}\\
		&\leq \beta'\frac{1}{(k-K_1+1)^{3\upsilon_1-1}}+\beta''\frac{1}{(k-K_1+1)^{3\upsilon_1-1}}\\
		&=(\beta'+\beta'')\frac{1}{(k-K_1+1)^{3\upsilon_1-1}},
	\end{split}
\end{equation*}
and we deduce that $\beta_k$ has a convergence rate of at least $\frac{1}{k^{3\upsilon_1-1}}$, which concludes the proof.

\subsection{Proof of Theorem \ref{special_case}}\label{special_case_proof}
Theorem \ref{cv_rate_th} indicates that the convergence rate is a function of $\upsilon_1$ and $\upsilon_2$, as $\gamma_k^2\propto (k+1)^{-2\upsilon_2}$ and $\frac{\alpha_k}{\gamma_k}\propto(k+1)^{-(-\upsilon_1-\upsilon_2)}$. Nonetheless, we must still verify the validity of the assumptions presented in the theorem, meaning:
\begin{itemize}
	\item Are $\sigma_1 < A$ and $\sigma_5 < A$ fulfilled?
	\item Are $\varsigma_1$ and $\varsigma_2$ bounded?
\end{itemize}
We must remark that in what follows, the analysis is done for $k\geq K_2$. 

Let $\alpha_k$ and $\gamma_k$ have the forms given in (\ref{step_sizes_eg}).
\begin{enumerate}
	\item \textbf{Verifying that} $\sigma_1 < A$ \textbf{and} $\sigma_5 < A$
	
	The idea is to find a bound on $\alpha_0$ and $\gamma_0$ to guarantee $\sigma_1<A$ and $\sigma_5<A$. We start by bounding $\sigma_1$ and $\sigma_5$ from above, i.e.,
	
	\begin{equation*}
		\begin{split}
			\sigma_1=\underset{k\geq K_2}{\max}\; \frac{1-(\frac{\gamma_{k+1}}{\gamma_{k}})^2}{\alpha_k\gamma_k} = \underset{k\geq K_2}{\max}\; \frac{1-(1+\frac{1}{k+1})^{-2\upsilon_2}}{\alpha_0\gamma_0(k+1)^{-\upsilon_1-\upsilon_2}}
		\end{split}
	\end{equation*}
	and
	\begin{equation*}
		\sigma_5 = \underset{k\geq K_2}{\max}\;\frac{1-\frac{\alpha_{k+1}\gamma_{k+1}^{-1}}{\alpha_{k}\gamma_{k}^{-1}}}{\alpha_{k}\gamma_k}=\underset{k\geq K_2}{\max}\;\frac{1-(1+\frac{1}{k+1})^{-(\upsilon_1-\upsilon_2)}}{\alpha_{0}\gamma_{0}(k+1)^{-\upsilon_1-\upsilon_2}}.
	\end{equation*}	
	To do so, we define a function $q(x)=x^{-a}(1-(1-x)^{-b})$ with $a, b, x \in (0, 1]$. Since $x^{-a}\leq x^{-1}$, we have $q(x)\leq x^{-1}(1-(1-x)^{-b})=r(x)$. To further bound $q(x)$, We study the derivative of $r(x)$ as it's simpler to do so,
	\begin{equation*}
		r'(x)=x^{-2}\bigg(((b+1)x+1)(1+x)^{-b-1}-1\bigg) = x^{-2}s(x).
	\end{equation*}
	Hence the sign of $r'(x)$ is that of $s(x)$. We again calculate the derivative of $s(x)$ to find its sign,
	\begin{equation*}
		s'(x) = -b(b+1)x(1 + x)^{-b-2}\leq0
	\end{equation*}
	since $b>0$ and $x>0$. Then, $s(x)$ is a decreasing function of	$x$ over $(0, 1]$. We remark that $\lim_{x\rightarrow 0} s(x)=0$, meaning $s(x)<0$ and $r'(x)<0$, $\forall x\in(0,1]$. Finally,
	\begin{equation*}
		r(x)<\lim_{x\rightarrow 0}r(x)=\frac{1-(1+x)^{-b}}{x}=b,
	\end{equation*}
	and $q(x)\leq r(x)<b$, noting that $\lim_{x\rightarrow 0}q(x)=b$ for $a = 1$.
	We conclude that $\sigma_1<\frac{2\upsilon_2}{\alpha_{0}\gamma_{0}}$ and $\sigma_1<\frac{\upsilon_1-\upsilon_2}{\alpha_{0}\gamma_{0}}$. For $\sigma_1 < A$ and $\sigma_5 < A$ to be valid, we must have 
	\begin{equation}\label{alphagamma}
		\alpha_0\gamma_0\geq \max\{2\upsilon_2,\upsilon_1-\upsilon_2\}/A.
	\end{equation}
	
	\item \textbf{Verifying that} $\varsigma_1$ \textbf{and} $\varsigma_2$ \textbf{are bounded}
	
	The goal is to verify that the constant term in the convergence rate is bounded. Thus, we must check that the lower bounds given in (\ref{vartheta}) and (\ref{varrho}) are indeed finite. We begin by analyzing $\sigma_2$ and $\sigma_7$, i.e.,
	\begin{equation*}
		\sigma_2 = \underset{k\geq K_2}{\max}\;\frac{\delta_k}{\gamma_k^{2}} =\underset{k\geq K_2}{\max}\;\gamma_{0}^{-2}(k+1)^{2\upsilon_2} \bigg(\frac{1+\rho_w^2}{2}\bigg)^k
	\end{equation*}
	and
	\begin{equation*}
		\sigma_7 = \underset{k\geq K_2}{\max}\;\frac{\gamma_k}{\alpha_{k}}\delta_k = \underset{k\geq K_2}{\max}\; \gamma_0\alpha_0^{-1}(k+1)^{\upsilon_1-\upsilon_2}\bigg(\frac{1+\rho_w^2}{2}\bigg)^k.
	\end{equation*}
	To prove that $\sigma_2$ and $\sigma_7$ are finite, we define the function $p(k)=a(k+1)^b\bigg(\frac{1+\rho_w^2}{2}\bigg)^k$, with $a>0$ and $0<b<1$. To find the maximum, we find $k$ such that $p'(k)=0$, meaning
	\begin{equation*}
		a(k+1)^b\bigg(\frac{1+\rho_w^2}{2}\bigg)^k\bigg(\frac{b}{k+1}+\ln\bigg(\frac{1+\rho_w^2}{2}\bigg)\bigg) = 0.
	\end{equation*}
	Hence, 
	\begin{equation*}
		k=b\ln^{-1}\bigg(\frac{2}{1+\rho_w^2}\bigg)-1.
	\end{equation*}
	\begin{itemize}
		\item From here, we can say if $b\ln^{-1}\bigg(\frac{2}{1+\rho_w^2}\bigg)-1\geq K_2$, then
		\begin{equation*}
			\underset{k\geq K_2}{\max}\;p(k)=\frac{a 	b^b}{\big(\frac{2}{1+\rho_w^2}\big)\ln^b\big(\frac{2}{1+\rho_w^2}\big)} e^{-b}<\infty.
		\end{equation*}
		\item Otherwise, if $b\ln^{-1}\bigg(\frac{2}{1+\rho_w^2}\bigg)-1< K_2 \leq k$, then
		$\frac{b}{k+1}+\ln\bigg(\frac{1+\rho_w^2}{2}\bigg)<0$ which gives
		$p'<0$ for $k\geq K_2$, meaning $p$ is strictly decreasing and
		\begin{equation*}
			\underset{k\geq 	K_2}{\max}\;p(k)=p(K_2)=a(K_2+1)^b\bigg(\frac{1+\rho_w^2}{2}\bigg)^{K_2} <\infty.
		\end{equation*}
	\end{itemize}
	We conclude that $\sigma_2<\infty$ and $\sigma_7<\infty$.
	
	Next, we study the finiteness of $\sigma_3$ and $\sigma_8$. From Lemma \ref{beta_k}, we can write
	\begin{equation*}
		\begin{split}
			\sigma_3 = \underset{k\geq K_2}{\max}\; \frac{\beta_k}{\gamma_k^{2}}&\leq \underset{k\geq K_2}{\max}\;(\beta'+ \beta'')\frac{\alpha_{k-K_1}^2}{\gamma_k^2}\\
			&= \underset{k\geq K_2}{\max}\;(\beta'+ \beta'')\alpha_0^2\gamma_0^{-2}\frac{(k+1)^{2\upsilon_2}}{(k-K_1+1)^{3\upsilon_1-1}}
		\end{split}
	\end{equation*}
	and 
	\begin{equation*}
		\begin{split}
			\sigma_8 = \underset{k\geq K_2}{\max}\;\frac{\gamma_k}{\alpha_{k}}\beta_k &\leq \underset{k\geq K_2}{\max}\;(\beta'+\beta'')\frac{\gamma_k}{\alpha_{k}}=\frac{\gamma_k\alpha_{k-K_1}^2}{\alpha_{k}}\\
			&=\underset{k\geq K_2}{\max}\;(\beta'+\beta'')\gamma_0\alpha_0\frac{(k+1)^{\upsilon_1-\upsilon_2}}{(k-K_1+1)^{3\upsilon_1-1}}.
		\end{split}
	\end{equation*}
	We then define a function $q(k)=a\frac{(k+1)^b}{(k-K_1+1)^{3\upsilon_1-1}}$ for $k\geq K_2$, with $a>0$ and $0<b<1$, and we study its derivative
	\begin{equation*}
		q'(k) = a\frac{(b-3\upsilon_1+1)k-bK_1+b-3\upsilon_1+1}{(k+1)^{1-b}(k-K_1+1)^{3\upsilon_1}}.
	\end{equation*}
	We know that $q'<0$, and thus $q$ is strictly decreasing for $k\geq K_2$, when $b-3\upsilon_1+1\leq 0$. 
	
	Hence, 
	\begin{equation}\label{sigma_3}
		\begin{split}
			\sigma_3 =\Bigg\{
			\begin{matrix}
				\alpha_0^2\gamma_0^{-2}\frac{(K_2+1)^{2\upsilon_2}}{(K_2-K_1+1)^{3\upsilon_1-1}},& \text{if}\;2\upsilon_2-3\upsilon_1+1\leq0,\\
				\infty,&\text{if}\;2\upsilon_2-3\upsilon_1+1>0,
			\end{matrix}
		\end{split}
	\end{equation}
	and
	\begin{equation*}
		\begin{split}
			\sigma_8 = (\beta'+\beta'')\gamma_0\alpha_0\frac{(K_2+1)^{\upsilon_1-\upsilon_2}}{(K_2-K_1+1)^{3\upsilon_1-1}}<\infty
		\end{split}
	\end{equation*}
	since $b-3\upsilon_1+1=-\upsilon_2-2\upsilon_1+1<0$ always holds.
	
	We end with the analysis of $\sigma_4$ and $\sigma_6$, i.e.,
	\begin{equation*}
		\sigma_4 = \alpha_0\gamma_0^{-3}\underset{k\geq K_2}{\max}\; (1+k)^{-(\upsilon_1-3\upsilon_2)}
		=\Bigg\{
		\begin{matrix}
			\alpha_0\gamma_0^{-3}(1+K_2)^{-(\upsilon_1-3\upsilon_2)},& \text{if}\;\upsilon_1\geq3\upsilon_2,\\
			\infty,&\text{if}\;\upsilon_1<3\upsilon_2,
		\end{matrix}
	\end{equation*}
	and
	\begin{equation*}
		\sigma_6 = \alpha_0^{-\frac{1}{2}}\gamma_0^{\frac{3}{2}}\underset{k\geq K_2}{\max}\;(1+k)^{\frac{\upsilon_1-3\upsilon_2}{2}}
		=\Bigg\{
		\begin{matrix}
			\alpha_0^{-\frac{1}{2}}\gamma_0^{\frac{3}{2}}(1+K_2)^{\frac{\upsilon_1-3\upsilon_2}{2}},& \text{if}\;\upsilon_1\leq3\upsilon_2,\\
			\infty,&\text{if}\;\upsilon_1>3\upsilon_2.
		\end{matrix}
	\end{equation*}
	There are clearly 3 cases:
	\begin{itemize}
		\item $\upsilon_1>3\upsilon_2$
		
		Thus, $\sigma_4$ is bounded. 
		
		Since now $2\upsilon_2-3\upsilon_1+1<\frac{2}{3}\upsilon_1-3\upsilon_1+1=-\frac{7}{3}\upsilon_1+1<0$ always holds, then $\sigma_3$ (\ref{sigma_3}) and $\varsigma_1$ (by definition) are also bounded provided that $\alpha_0\gamma_0\geq \frac{2\upsilon_2}{A}$ in (\ref{alphagamma}). 
		
		However, $\varsigma_2\rightarrow\infty$ since $\sigma_6\rightarrow \infty$ resulting in a loose upper bound in (\ref{rate_2}). 
		
		To that end, we can write $D_k\leq\Upsilon_1 (1+k)^{-2\upsilon_2}$ with  $\Upsilon_1$ a bounded constant.
		
		\item $\upsilon_1<3\upsilon_2$
		
		Similarly, $\sigma_6$ is bounded while $\sigma_4\rightarrow \infty$. Then, $\exists$ $\Upsilon_2<\infty$, where $D_k\leq\Upsilon_2(1+k)^{-(\upsilon_1-\upsilon_2)}$ provided that $\alpha_0\gamma_0\geq\frac{\upsilon_1-\upsilon_2}{A}$.
		
		\item $\upsilon_1=3\upsilon_2$
		
		Both $\sigma_4$ and $\sigma_6$ are bounded allowing both previous inequalities corresponding to $D_k$ to be valid.
	\end{itemize}
	By this analysis, we conclude the proof of Theorem \ref{special_case}.
\end{enumerate}
\section{Regret Analysis}\label{regret}
Since, by Lemma \ref{lipschitz}, the objective function is $L$-smooth, we can write 
\begin{equation}\label{Fsmooth}
	\mathcal{F}(y)\leq \mathcal{F}(x)+\langle\nabla\mathcal{F}(x), y-x\rangle+\frac{L}{2}\|y-x\|^2,\;\forall x,y\in\mathbb{R}^d.
\end{equation}
To find the regret bound, consider
\begin{equation*}
	\begin{split}
		\mathbb{E}_{\mathcal{H}_k}\bigg[\sum_{k=1}^{K}\mathcal{F}(\bar{x}_k)-\mathcal{F}(x^*)\bigg]
		&\overset{(a)}{\leq} \frac{L}{2} \sum_{k=1}^{K}\mathbb{E}_{\mathcal{H}_k}\bigg[ \|\bar{x}_k-x^*\|^2\bigg]\\
		&= \frac{L}{2}\sum_{k=1}^{K}D_k\\
		&\overset{(b)}{\leq} \Upsilon\frac{L}{2}\sum_{k=1}^{K} \frac{1}{\sqrt{k+1}}\\
		&\overset{(c)}{\leq} \Upsilon\frac{L}{2}\int_{0}^{K} \frac{1}{\sqrt{u+1}} du\\
		&= \Upsilon L (\sqrt{K+1}-1)\\
	\end{split}
\end{equation*}
where $(a)$ is due to $\nabla\mathcal{F}(x^*)=0$, by definition of $x^*$, in (\ref{Fsmooth}), $(b)$ is by Theorem \ref{special_case}, and in $(c)$ we perform the same interpretation of the monotonically decreasing function $\frac{1}{\sqrt{u+1}}$ as in (\ref{beta_sum2}).

\vskip 0.2in
\bibliography{ZO_One_Point_Estimate_with_Distributed_Stochastic_GT_Technique__Mhanna_Assaad}

\end{document}